\documentclass[12pt]{amsart}
\usepackage{amsmath,amssymb,amsthm,amsfonts,enumerate}

\usepackage{float}
\usepackage{graphicx}
\usepackage{setspace}
\usepackage{MnSymbol}
\usepackage{amscd}
\usepackage{bm}
\usepackage{hyperref}

\AtBeginDocument{}

\usepackage{enumitem}
\setdescription{leftmargin=0pt}

\newtheorem{thm}{Theorem}[section]
\newtheorem*{thm*}{Theorem}

\newtheorem{lem}[thm]{Lemma}
\newtheorem{cor}[thm]{Corollary}

\theoremstyle{definition}
\newtheorem{defi}{Definition}

\theoremstyle{definition}
\newtheorem{ex}{Example}

\theoremstyle{remark}
\newtheorem{rem}{Remark}

\title[Torsion normal generators of the mapping class group]{Torsion normal generators\\ of the mapping class group\\ of a non-orientable surface}
\author{Marta Le\'sniak}
\address{Institute of Mathematics, Faculty of Mathematics, Physics and Informatics, University of Gda\'nsk, 80-308 Gda\'nsk, Poland}
\email{marta.lesniak@mat.ug.edu.pl}
\thanks{Supported by National Science Centre, Poland, grant 2015/17/B/ST1/03235.}

\begin{document}

\begin{abstract}
We show that the normal closure of any periodic element of the mapping class group of a non-orientable surface whose order is greater than 2 contains the commutator subgroup, which for $g\geq 7$ is equal to the twist subgroup, and provide necessary and sufficient conditions for the normal closures of involutions to contain the twist subgroup. Finally, we provide a criterion for a periodic element to normally generate $\mathcal{M}(N_g)$ and give examples.
\end{abstract}

\maketitle

\section{Introduction}

In \cite{LM}, Lanier and Margalit provided a comprehensive look at the normal closures of elements of the mapping class group of an orientable surface. This paper aims to provide analogous results for periodic elements of a mapping class group of a non-orientable surface.

It is known that the mapping class group of a closed surface is generated by elements of finite order. This fact was first proved for an orientable surface by Maclachlan \cite{McL2}, who used the result in his proof of simple connectivity of the moduli space of Riemann surfaces. Since then many papers have been devoted to torsion generators of mapping class groups. For a closed orientable surface $S_g$ of genus $g\ge 3$, McCarthy-Papadopoulos \cite{MP} and Korkmaz \cite{MK2} found examples of an element $f\in\mathcal{M}(S_g)$ of finite order ($2$ and $4g+2$, respectively) whose normal closure in $\mathcal{M}(S_g)$ is the entire group. We say that such an $f$ {\it normally generates} $\mathcal{M}(S_g)$. The result of Lanier and Margalit is that any non-trivial periodic mapping class other than the hyperelliptic involution is a normal generator of $\mathcal{M}(S_g)$.

In the case of a closed non-orientable surface $N_g$ it is known that $\mathcal{M}(N_g)$ is generated by involutions \cite{Sz2}. The author of this paper together with Szepietowski \cite{LSz} have recently proved that $\mathcal{M}(N_g)$ is normally generated by one element of infinite order for $g\ge 7$.
 
While in the orientable case the mapping class group is generated by Dehn twists as shown by Dehn and Lickorish \cite{Lick}, in the non-orientable case Dehn twists generate a subgroup of index 2 known as the twist subgroup $\mathcal{T}(N_g)$ \cite{Lick1}. Similarly, where in the orientable case the commutator subgroup of the mapping class group $\mathcal{M}(S_g)$ is equal to the entire group for $g\geq 3$ \cite{Pow}, in our case the commutator subgroup of $\mathcal{M}(N_g)$ is equal to the twist subgroup for $g\geq 7$ \cite{MK}. If we use methods similar to the ones employed in \cite{LM}, the question we answer is whether the normal closure of a given element contains the twist subgroup of $N_g$. If the answer is positive, this still leaves two possibilities: the normal closure can be either the twist subgroup itself or the entire mapping class group. We learn to distinguish between these cases and indicate several finite order normal generators of $\mathcal{M}(N_g)$.

Another difference between the orientable and non-orientable cases is the abundance of involutions, that is, elements of order 2. Orientation-preserving involutions with fixed points on a given orientable surface form a 1-parameter family and are determined up to conjugacy by the genus of the quotient orbifold \cite{Dug}. On the contrary, involutions on non-orientable surfaces are determined by up to six invariants \cite{Dug} and we look closely at each case to formulate the sufficient and necessary conditions for their normal closures to contain the twist subgroup.

With that in mind, it is no surprise that our results come in two main cases: for periodic mapping classes of order greater than 2 and for involutions. But before we move onto the main results, a few words about notation.

Any periodic mapping class $f\in\mathcal{M}(N_g)$ is represented by a homeomorphism whose order is equal to that of $f$; this was proved by Kerckhoff \cite{Kerck} in the orientable case and in the non-orientable case by applying double covering by an orientable surface (see the last two paragraphs of Section IV, page 256, in \cite{Kerck}). Moreover, this homeomorphism can be chosen to be an isometry with respect to some hyperbolic metric on $N_g$ and is unique up to conjugacy in the group of homeomorphisms of $N_g$. We refer to any such representative as a standard representative of $f$ and denote it by $\phi$.

The set of fixed points of $\phi$ consists of isolated fixed points and point-wise-fixed simple closed curves we will refer to as ovals. Note that ovals appear only in the case of involutions; elements of greater finite order only have isolated fixed points. Let $r$ denote the cardinality of the set of isolated fixed points of $\phi$, $k$ -- the cardinality of the set of ovals, $k_{+}$ -- two-sided ovals and $k_{-}$ -- one-sided ovals.

Now we are ready to formulate the main results.

\begin{thm}\label{noninv}
	Let $g\geq 5$ and let $f$ be a periodic element of $\mathcal{M}(N_g)$ of order greater than 2. The normal closure of $f$ contains the commutator subgroup of $\mathcal{M}(N_g)$.
	
\end{thm}

\begin{cor}\label{ng}
Let $g\geq 7$ and let $f$ be a periodic element of $\mathcal{M}(N_g)$ of order greater than 2. The normal closure of $f$ is either $\mathcal{T}(N_g)$ or $\mathcal{M}(N_g)$, the latter if and only if $f\notin \mathcal{T}(N_g)$.
\end{cor}

\begin{rem}
	
	Thanks to Corollary \ref{ng}, we can give an example of a torsion normal generator of $\mathcal{M}(N_g)$ for any $g\geq 7$. Let us represent $N_g$ as a sphere with $g$ crosscaps. If $g$ is even, one element normally generating $\mathcal{M}(N_g)$ is a rotation of order $g$ of the sphere with $g$ crosscaps spaced evenly along the equator. If $g$ is odd, one such element is a rotation of order $g-1$ of the sphere with $g-1$ crosscaps spaced evenly along the equator and one stationary crosscap at one of the poles (see Figure \ref{figexamples}).
\sloppy
By computing the determinant of the induced automorphism of $H_1(N_g;\mathbb{R})$ it can be shown that these elements are not in the twist subgroup (see Theorem \ref{thm:twsb} in Section \ref{sec:prem}). Details are left as an exercise to the reader.
	
\end{rem}

\begin{figure}[!htbp]\begin{center}
\includegraphics[scale=.35]{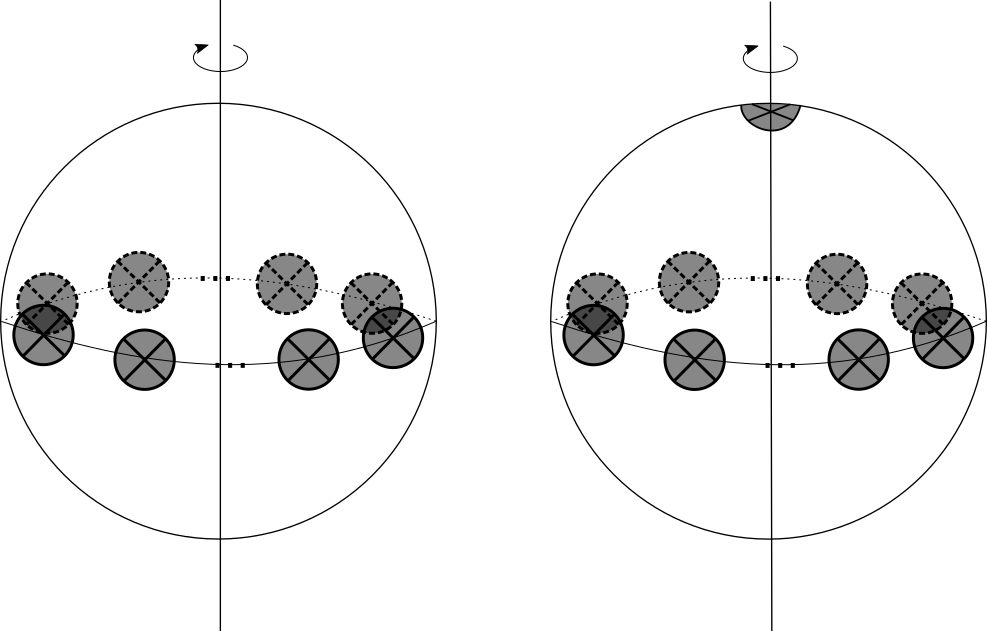}
\caption{Torsion normal generators of $\mathcal{M}(N_g)$ for even and odd $g$, respectively.}
\label{figexamples}\end{center}\end{figure}

\begin{rem}
For $2\leq g\leq 6$ the group $\mathcal{M}(N_g)$ is not normally generated by one element, because its abelianisation is not cyclic \cite{MK}.

\end{rem}

\begin{thm}\label{inv}
	Let $g\geq 5$ and let $f\in\mathcal{M}(N_g)$ be an involution, $\phi$ its standard representative and $r$, $k$, $k_-$, $k_+$ the parameters defined above.
	The normal closure of $f$ in $\mathcal{M}(N_g)$ contains the commutator subgroup of $\mathcal{M}(N_g)$ if and only if one of the following conditions hold.

	\begin{enumerate}

		\item $r>0$, $k=0$ and $g-r\geq 4$.
		
		\item $k>0$, $r+k_->0$ and either:
		
		\begin{enumerate}
			\item the orbifold $N_g/\left\langle \phi \right\rangle$ is non-orientable, or
			\item the orbifold $N_g/\left\langle \phi \right\rangle$ is orientable and $g-r-2k\geq 2$.
		\end{enumerate}
		\item $r=k_{-}=0$ and either
		
		\begin{enumerate}
			\item the set of fixed points of $\phi$ is separating and $g-2k\geq 4$, or
			
			\item the set of fixed points of $\phi$ is non-separating.
			
		\end{enumerate}
		
	\end{enumerate}

\end{thm}

\begin{thm}\label{thm:normgens}
	For any $g\geq 7$ there is an involution which normally generates $\mathcal{M}(N_g)$.
\end{thm}

Our results can be interpreted in terms of covers of the moduli space of non-orientable Klein surfaces. A Klein surface is a surface with a dianalytic structure \cite{AG}. Similarly as in the case of Riemann surfaces, the moduli space $\mathfrak{M}_g$ of Klein surfaces homeomorphic to $N_g$ can be defined as the orbit space of a properly discontinuous action of $\mathcal{M}(N_g)$ on a Teichm\"uller space (see \cite{Sz2} and references therein). The moduli space $\mathfrak{M}_g$ is an orbifold whose orbifold fundamental group is  $\mathcal{M}(N_g)$, and whose orbifold points correspond to periodic elements of $\mathcal{M}(N_g)$. Since normal subgroups of $\mathcal{M}(N_g)$ correspond to regular orbifold covers of  $\mathfrak{M}_g$, Theorem \ref{noninv} says that any such cover  of degree greater than $2$ can have orbifold points only of order $2$.

This paper is organised as follows: in \nameref{sec:prem} we review basic concepts about mapping class groups, in Section \ref{sec:stds} we define standard pairs of curves and show how they relate to the main theorems. Afterwards, in Section \ref{sec:nec}, we introduce the language of NEC groups in actions of finite groups on surfaces. We use the notions introduced so far to prove Theorem \ref{noninv} in Section \ref{sec:proofnoninv}, first in the case when the action of $\left\langle \phi \right\rangle$ on $N_g$ is free, then otherwise. In Section \ref{sec:class} we discuss how to construct involutions on surfaces and use the introduced surgeries to list all possible involutions on a surface $N_g$, based on the work of Dugger \cite{Dug}. Finally, in Section \ref{sec:proofinv} we demonstrate how to show that the normal closure of a map does not contain the twist subgroup using the induced action on homology groups with coefficients in $\mathbb{Z}_2$ and prove Theorems \ref{inv} and \ref{thm:normgens}.

\section{Preliminaries}\label{sec:prem}

The non-orientable surface $N_g$ can be represented as an orientable surface of genus $l$ with $s$ crosscaps, where $2l+s=g$ and $s\geq 1$. In all figures of this paper shaded crossed discs represent crosscaps, meaning that the interiors of those discs are removed and antipodal points on their boundaries identified.

The \emph{mapping class group} of the surface $N_g$, $\mathcal{M}(N_g)$, is the quotient of the group of all self-homeomorphisms of $N_g$ by the subgroup of self-homeomorphisms isotopic to the identity.

A \emph{curve} on a surface is a simple closed curve. By an abuse of notation we do not distinguish a curve from its isotopy class or image. If the regular neighbourhood of a curve is an annulus, we call it two-sided, and if it is a M\"{o}bius strip -- one-sided. For $c,d$ -- two curves we take $i(c,d)$ to be their geometric intersection number, that is, $i(c,d)=\min\{|\gamma\cap\delta|: \gamma\in c, \delta\in d\}$.

About a two-sided curve $a$ on $N_g$ we define a Dehn twist $T_a$ (the mapping and the mapping class). Because we are working on a non-orientable surface, the Dehn twist $T_a$ can be one of two, depending on which orientation of a regular neighbourhood of $a$ we choose. Then the Dehn twist in the opposite direction is $T_a^{-1}$. For any $f\in\mathcal{M}(N_g)$ we have

\begin{equation*}
fT_cf^{-1}=T_{f(c)}^\epsilon
\end{equation*}
where $\epsilon\in\{-1,1\}$, again depending on which orientation of a regular neighbourhood of $f(c)$ we choose. The subgroup of $\mathcal{M}(N_g)$ generated by all Dehn twists we call the twist subgroup and denote by $\mathcal{T}(N_g)$. The following lemma is proven in \cite{Sz1}:

\begin{lem}\label{commute}
For $g\geq 5$, let $c$, $d$ be two-sided simple closed curves on $N_{g}^{n}$ such that $i(c,d)=1$ and let $f:\mathcal{M}(N_g)\rightarrow G$ be a homomorphism. If the elements $f({T_c})$ and $f(T_d)$ commute in $G$, the image of $\mathcal{M}(N_g)$ under $f$ is abelian.
\end{lem}

\begin{cor}\label{cor}
Let $c$, $d$ and $f$ be as in the lemma above. The kernel of $f$ contains the commutator subgroup of $\mathcal{M}(N_g)$.
\end{cor}

\begin{thm}\label{twist}\cite{MK}
For $g\geq 7$, $[\mathcal{M}(N_{g}),\mathcal{M}(N_{g})]=\mathcal{T}(N_{g})$. For $g=5,6$ $[\mathcal{M}(N_{g}),\mathcal{M}(N_{g})]$ has index 4 in $\mathcal{M}(N_{g})$.
\end{thm}

Now consider the action of $\mathcal{M}(N_g)$ on $H_1(N_g;\mathbb{R})$. For $f\in\mathcal{M}(N_g)$ and $f_*:H_1(N_g;\mathbb{R})\rightarrow H_1(N_g;\mathbb{R})$ the induced homomorphism, we define the \emph{determinant homomorphism} $D:\mathcal{M}(N_g)\rightarrow\{-1,1\}$ as $D(f)=\det(f_*)$.

\begin{thm}\label{thm:twsb}\cite{MS}
For $f\in \mathcal{M}(N_g)$, $\det(f_*)=1$ if and only if $f\in\mathcal{T}(N_g)$.
\end{thm}
We immediately conclude that, provided the normal closure of $f$ in $\mathcal{M}(N_g)$ contains the twist subgroup, the normal closure of $f$ is equal to $\mathcal{M}(N_g)$ if and only if $\det(f_*)=-1$. This fact allows us to look for normal generators of $\mathcal{M}(N_g)$.

\section{Standard pairs of curves}\label{sec:stds}

The following two lemmas are adaptations of Lemmas 2.1, 2.2, 2.3 proven in \cite{LM} for orientable surfaces with minor changes in assumptions and proofs.

\begin{lem}\label{L2.1}

For $g\geq 5$, let $c$ and $d$ be two-sided non-separating curves in $N_{g}$ with $i(c,d)=1$. Then the normal closure of $T_{c}T_{d}$ in $\mathcal{M}(N_{g})$ is equal to the commutator subgroup of $\mathcal{M}(N_{g})$.

\end{lem}

\begin{proof}

Since $N\setminus c$ and $N\setminus d$ are both connected and non-orientable, there exists an element $h\in \mathcal{M}(N_{g})$ such that $h(c)=d$. Then $h T_c h^{-1}=T_d^{\pm 1}$. Because $T_d$ is conjugate with $T_d^{-1}$ in $\mathcal{M}(N_{g})$ (Lemma 2.4 in \cite{MK}), we can take $h$ to be such that $h T_c h^{-1}=T_d^{-1}$. It follows that the element $T_{c}T_{d}=T_c h T_c^{-1} h^{-1}$ lies in the commutator subgroup. Since the commutator subgroup is a normal subgroup, it contains the normal closure of $T_{c}T_{d}$.

Now let $H$ denote the normal closure of $T_{c}T_{d}$ in $\mathcal{M}(N_{g})$ and let $p:\mathcal{M}(N_{g})\rightarrow\mathcal{M}(N_{g})/H$ be the canonical projection map. It is easy to see that $p(T_{c})=p(T_{d})^{-1}$; in particular, $p(T_{c})$ and $p(T_{d})$ commute. By Corollary \ref{cor} the kernel of $p$ contains the commutator subgroup of $\mathcal{M}(N_{g})$. Since $\ker p=H$, the proof is complete.
\end{proof}

\begin{defi}
Let $c$, $d$ be a pair of non-separating two-sided simple closed curves on a surface $N_{g}$. We will say that $c$ and $d$ form
\begin{itemize}
\item a type 1 standard pair if $i(c,d)=1$;
\item a type 2 standard pair if $i(c,d)=0$ and the complement $N_{g}\setminus(c\cup d)$ is connected and non-orientable.
\end{itemize}
\end{defi}

\begin{lem}\label{L2.2}
For $g\geq5$, let $f\in\mathcal{M}(N_{g})$. Suppose that there is a curve $c$ in $N_{g}$ such that $(c,f(c))$ form a standard pair of type 1 or 2. Then the normal closure of $f$ in $\mathcal{M}(N_{g})$ contains the commutator subgroup of $\mathcal{M}(N_{g})$.
\end{lem}

\begin{proof}
Let $T_c$ and $T_{f(c)}$ be such that $T_{f(c)}=fT_c^{-1}f^{-1}$.

Suppose first that $c$ and $f(c)$ form a type 1 standard pair.
Since the commutator $[T_{c},f]$ is equal to the product of $T_{c}fT_{c}^{-1}$ and $f^{-1}$, it lies in the normal closure of $f$. On the other hand, $T_{c}fT_{c}^{-1}f^{-1}$ is equal to $T_{c}T_{f(c)}$. Since $i(c,f(c))=1$, the lemma follows by Lemma \ref{L2.1}.

Now suppose that $c$ and $f(c)$ form a type 2 standard pair. Then we can find a two-sided, non-separating curve $d$ such that $i(c,d)=1$ and $(d,f(c))$ form a type 2 standard pair. As before, the commutator $[T_{c},f]$ is equal to $T_{c}T_{f(c)}^{-1}$ and lies in the normal closure of $f$. The group $\mathcal{M}(N_{g})$ acts transitively on type 2 standard pairs of curves, therefore there exists a mapping class $h\in\mathcal{M}(N_{g})$ such that $h(c)=f(c)$ and $h(f(c))=d$. We choose $T_d$ be to such that $h[T_{c},f]h^{-1}$ is equal to either $T_{f(c)}^{-1}T_{d}$ or $T_{f(c)}T_{d}^{-1}$. In the former case we have $T_cT_d=\left(T_{c}T_{f(c)}\right)\left(T_{f(c)}^{-1}T_{d}\right)$, and in the latter $T_cT_d=\left(T_{c}T_{f(c)}\right)\left(T_{f(c)}T_{d}^{-1}\right)^{-1}$ (note that because $d$ and $f(c)$ are disjoint, $T_d$ and $T_{f(c)}$ commute). In both cases $T_cT_d$ lies in the normal closure of $f$. The lemma follows by Lemma \ref{L2.1}.
This completes the proof.
\end{proof}
\begin{rem}
	It is easily seen that for any $f\in\mathcal{M}(N_g)$ and any $n\in\mathbb{N}$ the normal closure of $f^n$ is contained in the normal closure of $f$.
\end{rem}

\section{NEC groups}\label{sec:nec}

\subsection{NEC group and its fundamental polygon}\label{polynot}

Non-euclidean crystallographic (NEC) groups are discrete and cocompact subgroups of the group of isometries of the hyperbolic plane, $\mathrm{Isom}(\mathbb{H}^2)$. They provide a natural tool for studying finite group actions on surfaces; see for example \cite{JPAA,GMJ,RACSAM}.

Every action of a finite group $G$ on a surface $N_g$ can be obtained by means of a pair of NEC groups $\Gamma$ and $\Lambda$, with $\Gamma$ -- a normal torsion-free subgroup of $\Lambda$ such that $N_g$ is homeomorphic to $\mathbb{H}^2/\Gamma$ and $G$ is isomorphic to $\Lambda/\Gamma$. Equivalently, there is an epimorphism $\theta:\Lambda\rightarrow G$ with $\Gamma=\ker\theta\cong\pi_1(N_g)$.

The signature of an NEC group $\Lambda$ is a collection of non-negative integers and symbols. In our case, that is for cyclic $G$, the signature has the form
\begin{equation*}
	\left( h; \pm ; [m_1,...,m_r] ; \{()^k\} \right)
	\end{equation*}
where
	\begin{enumerate}
		
		\item the sign $\pm$ is ''$+$'' if $\mathbb{H}^2/\Lambda$ is orientable and ''$-$'' otherwise;
		
		\item the integer $h\geq0$ denotes the genus of $\mathbb{H}^2/\Lambda$;
		
		\item the ordered set of integers $m_1,...,m_r (m_i\geq 2)$, called the \emph{proper periods} of the signature, corresponds to cone points on the orbifold $\mathbb{H}^2/\Lambda$;
		
		\item $k$ empty \emph{period-cycles} $(),...,()$ correspond to boundary components of the orbifold $\mathbb{H}^2/\Lambda$.
		
	\end{enumerate}
\begin{rem}
A general NEC signature can also have nonempty period cycles \cite{Buj}.
\end{rem}	
	
	If the set of periods or the set of period cycles is empty, we write the brackets with no symbols between them. For example, the signature $\left(g; -; [];\{\} \right)$ has no proper periods and no period cycles; it is the signature of the NEC group isomorphic to $\pi_1(N_g)$. If we know all $m_i$ to have the same value $p$, we will write $[(p)^r]$.

The quotient orbifold can be reconstructed from the associated NEC group by identifying the right edges of a marked polygon which is the fundamental region $P$ of $\Lambda$, as detailed in \cite{Mac}. The marked polygon is a plane polygon in which certain edges are related by homeomorphisms; the edges are identified if one is the image of the other under an element of $\Lambda$. If the first edge has vertices, in order as we read the labels anticlockwise, $P$ and $Q$, and the other has vertices $R$ and $S$, the identifying homeomorphism can map $P$ on $S$ and $Q$ on $R$, pairing the edges orientably, or map $P$ on $R$ and $Q$ on $S$, pairing the edges nonorientably. Two sides paired orientably will be indicated by the same letter and a prime, for example $\xi,\xi'$; two sides paired nonorientably will be written using the same letter and an asterisk, for example $\alpha, \alpha^*$. If we mark all the egdes of the polygon accordingly and then write them in the order in which they appear around the polygon anticlockwise, we obtain the surface symbol of the polygon \cite{Mac}. The marked polygon of the signature
\begin{equation*}
\left( h; + ; [m_1,...,m_r] ; \{ ()^k \} \right)
\end{equation*}
has the surface symbol
\begin{equation*}
\xi_1\xi_1'...\xi_r\xi_r'\epsilon_1\gamma_{1}\epsilon_1'...\epsilon_k\gamma_{k}\epsilon_k'\alpha_1\beta_1'\alpha_1'\beta_1...\alpha_h\beta_h'\alpha_h'\beta_h
\end{equation*}
while the marked polygon of the signature
\begin{equation*}
\left( h; - ; [m_1,...,m_r] ; \{ ()^k \} \right)
\end{equation*}
has the surface symbol
\begin{equation*}
\xi_1\xi_1'...\xi_r\xi_r'\epsilon_1\gamma_{1}\epsilon_1'...\epsilon_k\gamma_{k}\epsilon_k'\alpha_1\alpha_1^*...\alpha_h\alpha_h^*.
\end{equation*}
The signature gives us a presentation of the NEC group, as shown by Wilkie in \cite{Wilkie}. The presentation is as follows:
\begin{description}
\item[Generators]
\begin{enumerate}
	\item $x_1,...,x_r$ (elliptic elements);
	\item $c_{1},...,c_{k}$ (hyperbolic reflections);
	\item $e_1,...,e_k$ (hyperbolic elements except for cases $h=0$ and $r=k=1$, when they are elliptic elements);
	\item \begin{enumerate}
		\item $a_1,b_1,...,a_h,b_h$ (hyperbolic elements) if the sign is $+$;
		\item $d_1,...,d_h$ (glide reflections) if the sign is $-$.
	\end{enumerate}
\end{enumerate}

\item[Relations]

\begin{enumerate}
	\item $x_i^{m_i}=1$, $i=1,...,r$;
	\item $c_{j}^2=1$, $j=1,...,k$;
	\item $c_{j}=e_j^{-1}c_{j}e_j$ $j=1,...,k$;
	\item The long relation: \begin{enumerate}
		\item $x_1...x_re_1...e_ka_1b_1a_1^{-1}b_1^{-1}...a_hb_ha_h^{-1}b_h^{-1}=1$ if the sign is $+$;
		\item $x_1...x_re_1...e_kd_1^2...d_h^2=1$ if the sign is $-$.
	\end{enumerate}
	
\end{enumerate}
\end{description}

As mentioned above, the marked polygon $P$ is also the fundamental region of the NEC group $\Lambda$ associated with it. The generators of $\Lambda$ map the edges of the marked polygon in the following way:
\begin{enumerate}
	\item $x_i(\xi_i')=\xi_i$, $i=1,...,r$;
	\item $e_j(\epsilon_j')=\epsilon_j$, $j=1,...,k$;
	\item $c_j$, $j=1,...,k$ is a reflection along the axis containing $\gamma_j$;
	\item \begin{enumerate}
		\item $a_l(\alpha_l')=\alpha_l$ and $b_l(\beta_l')=\beta_l$, $l=1,...,h$ if the sign is $+$;
		\item $d_l(\alpha_l^*)=\alpha_l$, $l=1,...,h$ if the sign is $-$.
	\end{enumerate}
\end{enumerate}
Then $\mathbb{H}^2/\Lambda \cong P/\sim$, where $\sim$ refers to identification of edges paired by the above homeomorphisms. Only generators $c_{1},...,c_{k}$ and $d_1,...,d_h$ are orientation-reversing, the others are all orientation-preserving.
\begin{lem}\label{nonorient}\cite{Buj,RACSAM}
	Suppose that $\Lambda$ is an NEC group with signature \[\left( h; \pm ; [m_1,...,m_r] ; \{()^k\}\right).\] A group homomorphism $\theta:\Lambda\rightarrow G$ defines an action of G on a non-orientable surface if and only if
	\begin{enumerate}
		\item $\theta(x_i)$ has order $m_i$ for $1\leq i\leq r$,
		\item $\theta(c_j)$ has order 2 for $1\leq j\leq k$, and
		\item $\theta(\Lambda^+)=G$, where $\Lambda^+$ is the subgroup of $\Lambda$ consisting of orientation-preserving elements.
	\end{enumerate}
\end{lem}
Notice that the first two conditions guarantee that $\Gamma=\ker \theta$ is torsion-free and the third condition ensures that $\Gamma$ contains orientation-reversing elements and is therefore the fundamental group of a non-orientable surface.

If $\Lambda/\Gamma$ has order $n$ and $\Gamma\cong\pi_1(N_g)$, the Hurwitz-Riemann formula takes the following form
\begin{equation*}
g-2=n\left(\epsilon h + k - 2+\sum_{i=1}^{r}\left(1-\frac{1}{m_i}\right)\right)
\end{equation*}
where $\epsilon$ is equal to $1$ if $N_g/\left\langle \phi \right\rangle$ is non-orientable and $2$ otherwise.

\subsection{Topological equivalence}\label{topeq}

Suppose $\theta_i:\Lambda\rightarrow G$ are two epimorphisms with $\Gamma=\ker\theta_i \cong\pi_1(N_g)$ for $i=1,2$. We say that $\theta_1$ and $\theta_2$ are topologically conjugate if and only if the corresponding $G$-actions are conjugate by a homeomorphism of $N_g$. Equivalently, $\theta_1$ and $\theta_2$ are topologically conjugate if and only if there exist automorphisms $\psi\in \mathrm{Aut}(\Lambda)$ and $\chi\in\mathrm{Aut}(G)$ such that $\chi\circ \theta_1=\theta_2\circ \psi$; see Definition 2.2 in \cite{RACSAM} and the preceeding remark, or Proposition 2.2 in \cite{JPAA}.

We list the automorphisms of $\Gamma$ which will be used in the proof of Theorem \ref{noninv}; these and more are given in \cite{GMJ} (see also references therein). If the sign is "$+$", we are going to use the following automorphisms of $\Gamma$:

$\sigma$ defined by $\sigma(x_r)=Ea_1^{-1}E^{-1}x_rEa_1E^{-1}$, $\sigma(a_1)=[a_1^{-1},E^{-1}x_r^{-1}E]a_1$, $\sigma(b_1)=b_1a_1^{-1}
E^{-1}x_rEa_1$, where $E=e_1...e_k$, and the identity on the remaining generators.

$\pi$ defined by $\pi(e_k)=a_1^{-1}e_ka_1$, $\pi(c_k)=a_1^{-1}c_ka_1$, $\pi(a_1)=[a_1^{-1},e_k^{-1}]a_1$, $\pi(b_1)=b_1a_1^{-1}e_ka_1$ and the identity on the remaining generators.

$\omega$ defined by $\omega(a_1)=a_1$, $\omega(b_1)=b_1a_1$ and the identity on the remaining generators.

If the sign is "$-$", we are going to use the following automorphisms of $\Gamma$:

$\gamma$ defined by $\gamma(d_1)=E^{-1}x_rEd_1$, $\gamma(x_r)=x_rEd_1E^{-1}x_r^{-1}Ed_1^{-1}E^{-1}x_r^{-1}$, where $E=e_1...e_k$, and the identity on the remaining generators.

$\epsilon$ defined by $\epsilon(d_1)=e_kd_1$, $\epsilon(e_k)=e_kd_1e_k^{-1}d_1^{-1}e_k^{-1}$, $\epsilon(c_k)=e_kd_1c_kd_1^{-1}e_k^{-1}$ and the identity on the remaining generators.

Regardless of sign, we are also going to use the following automorphisms of $\Gamma$:

$\rho_i$ defined by $\rho_i(x_i)=x_ix_{i+1}x_i^{-1}$, $\rho_i(x_{i+1})=x_i$ and the identity on the remaining generators.

$\lambda_j$ defined by $\lambda_j(e_j)=e_je_{j+1}e_j^{-1}$, $\lambda_j(e_{j+1})=e_j$, $\lambda_j(c_j)=e_jc_{j+1}e_j^{-1}$ and the identity on the remaining generators.

In \cite{GMJ}, the following two theorems are proven that provide criteria for topological equivalence of $\mathbb{Z}_p$-actions:
\begin{thm}\label{thm1}
	Suppose that $p$ is an odd prime, $\Lambda$ is an $NEC$ group of signature\ $(h;-;[(p)^r];\{-\})$ and $\theta_i:\Lambda\rightarrow \mathbb{Z}_p$ for $i=1,2$ are two epimorphisms with $\ker\theta_i$ isomorphic to the fundamental group of a non-orientable surface. Then $\theta_1$ and $\theta_2$ are topologically conjugate if and only if $(\theta_2(x_1),...,\theta_2(x_r))$ is a permutation of $(\epsilon_1a\theta_1(x_1),...,\epsilon_ra\theta_1(x_r))$ for some $a\in\{1,...,p-1\}$ and $\epsilon_j\in\{1,-1\}$, $j=1,...,r$.
\end{thm}

\begin{thm}\label{thm2}
	Suppose that $\Lambda$ is an $NEC$ group of signature\linebreak $(h;-;[(2)^r];\{()^k\})$ and $\theta_i:\Lambda\rightarrow \mathbb{Z}_2$ for $i=1,2$ are two epimorphisms with $\ker\theta_i$ isomorphic to the fundamental group of a non-orientable surface. Let
	\begin{equation*}
	k_-^{(i)}=\#\{j\in\{1,...,k\}|\theta_i(e_j)=1\}
	\end{equation*}for i=1,2.
	Then $\theta_1$ and $\theta_2$ are topologically conjugate if and only if
	\begin{enumerate}
		\item $k_-^{(1)}=k_-^{(2)}$, and if $r=k_-^{(1)}=k_-^{(2)}=0$ and the sign is '$-$' also
		\item $\theta_1(d_1...d_g)=\theta_2(d_1...d_g)$ and
		\item $\theta_1(d_1)=...=\theta_1(d_g)=0$ if and only if $\theta_2(d_1)=...=\theta_2(d_g)=0$.
	\end{enumerate}
\end{thm}

Let $\gamma_j$ be the axis of the reflection $c_j$ corresponding to a boundary component of $\mathbb{H}^2/\Lambda$. Notice that $\gamma_j$ projects to a two-sided circle on $N_g\cong\mathbb{H}^2/\Gamma$ if and only if $e_j\in\Gamma$. In other words, the invariant $k_-^{(i)}$ from Theorem \ref{thm2} is the number of one-sided ovals.

\section{Periodic elements of order greater than 2}\label{sec:proofnoninv}

\begin{proof}[Proof of Theorem \ref{noninv}]

Let $f$ be a periodic element of $\mathcal{M}(N_g)$ of order $\#f>2$ and $\phi$ its standard representative. We will  show that there exists a curve $c$ on $N_g$ such that $(c,f(c))$ is a standard pair. Theorem \ref{noninv} will then follow by Lemma \ref{L2.2}.

\textbf{Case 1:} the action of $\left\langle\phi \right\rangle$ is free. By raising $f$ to an appropriate power, we can assume that $\#f$ is equal to a prime $p$. Since the action of $\left\langle \phi \right\rangle$ is free, it is a covering space action.

Let $p$ be an odd prime. By Theorem \ref{thm1} the covering map is uniquely determined up to powers and conjugacy by the genus $h$ of $N_g/\left\langle \phi \right\rangle$. Thus we can assume that $\phi$ is a rotation by the angle $2\pi/p$ of the surface $N_g$, as in Figure \ref{figfreeh}. If $h\geq 3$, we can find a curve $c$ such that $c$ and $f(c)$ form a type 1 standard pair.
\begin{figure}[!htbp]\begin{center}
				\includegraphics[scale=.7]{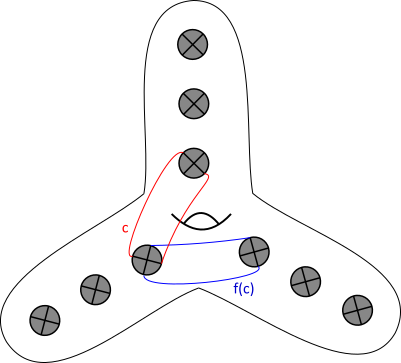}
				\caption{Type 1 standard pair for $h\geq 3$ and $p=3$.}
				\label{figfreeh}\end{center}
		\end{figure}		
		An easy argument from the Hurwitz-Riemann formula shows that a case where $h<3$ does not occur.
		
Now let $p=2$. By Theorem 4.1 of \cite{Dug}, there are exactly two distinct conjugacy classes of free actions of the cyclic group of order 2 on a non-orientable surface of even genus equal at least 4. For $g=2s$, $s\geq 2$ the two actions can be represented as the antipodism of a sphere with $2(s-1)$ crosscaps forming $s-1$ antipodic pairs and the antipodism of a torus with $2(s-2)$ crosscaps forming $s-2$ antipodal pairs; we will denote them $f_{01}$ and $f_{02}$, respectively.
		\begin{figure}[!htbp]\begin{center}
				\includegraphics[scale=.4]{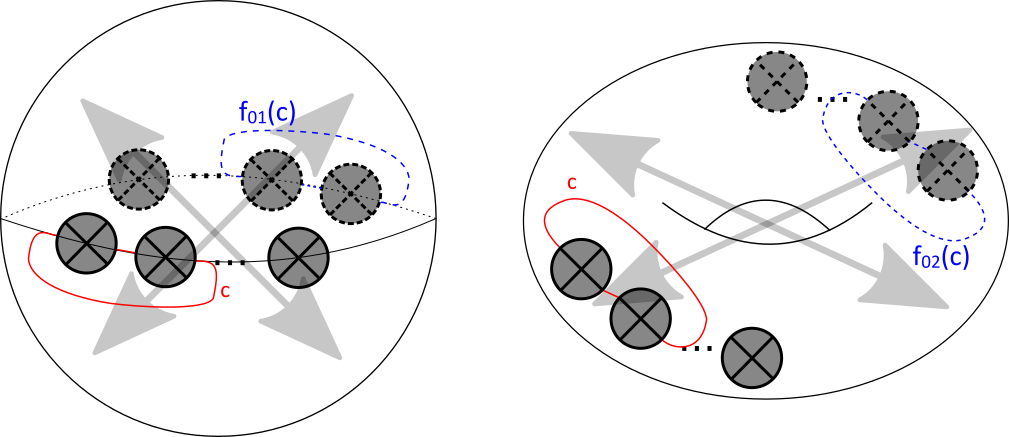}
				\caption{Free involutions.}
				\label{figfree2}\end{center}
		\end{figure}
		Since $g\geq 8$, we can always find a curve $c$ such that either $c$ and $f_{01}(c)$ or $c$ and $f_{02}(c)$ form a type 2 standard pair (Figure \ref{figfree2}). This completes the proof in Case 1.
		
For the rest of the proof we assume that the action of $\left\langle\phi \right\rangle$ is not free. We will use the NEC groups introduced in Section \ref{sec:nec}. For a cyclic group $\left\langle \phi \right\rangle$ acting on a non-orientable surface $N_g$ there exist NEC groups $\Lambda$, $\Gamma$ such that $\Gamma\triangleleft\Lambda$ and $\Gamma$ is torsion-free, $\left\langle \phi \right\rangle\cong \Lambda/\Gamma$, $N_g\cong \mathbb{H}^2/\Gamma$ and $N_g/\left\langle \phi \right\rangle\cong \mathbb{H}^2/\Lambda$. We identify $\Lambda/\Gamma$ with $\mathbb{Z}_n$ and let $\theta:\Lambda\rightarrow\mathbb{Z}_n$ be the canonical projection. Our goal is to find a curve $c$ on $\mathbb{H}^2/\Gamma$ and a generator $y\in \Lambda/\Gamma$ such that $c$ and $y(c)$ form a standard pair of either type. Then Theorem \ref{noninv} will follow from Lemma \ref{L2.2}. To achieve this, we construct the fundamental region $D$ of $\Gamma$ from the fundamental region $P$ of $\Lambda$, namely, $D=P\cup \tilde{y}P\cup...\cup \tilde{y}^{n-1}P$ for $n=\#\phi$, where $y=\tilde{y}\Gamma$. After we identify the edges of $D$ paired by elements of $\Gamma$, we obtain the surface $N_g$.

	\textbf{Case 2:} $\#f$ is not a power of 2. By raising $f$ to an appropriate power we can assume that $\# f$ is an odd prime. The group $\Lambda$ has no reflections, as there are no elements of order 2 in $\mathbb{Z}_p$ for them to map onto. Moreover, $m_i=p$ for $i=1,...,r$. The signature of $\Lambda$ is
\[
\left(h;-;[(p)^r]);\{-\}\right).
\]
	The marked polygon $P$ associated with $\Lambda$ has the surface symbol
	\begin{equation*}
	\xi_1\xi_1'...\xi_r\xi_r'\alpha_1\alpha_1^*...\alpha_h\alpha_h^*.
	\end{equation*}	
	Notice that $r\geq 1$, as we assumed the action is not free, and $h\geq 1$. The element $x_1$ is of order $p$, therefore it does not lie in the kernel $\Gamma$; we can assume $\theta(x_1)=1$. We can express $\Lambda$ as a sum of cosets of $\Gamma$ in the following way:
	\begin{equation*}
	\Lambda=\Gamma\cup x_1\Gamma\cup...\cup x_1^{p-1}\Gamma.
	\end{equation*}
	From this we conclude that the polygon $D$ which is the fundamental region for $\Gamma$ can be expressed as
	\begin{equation*}
	P\cup x_1P\cup...\cup x_1^{p-1}P=D
	\end{equation*}
	We take $y=x_1\Gamma$ as our generator of $\Lambda/\Gamma$.
	
	If $h\geq 2$, then by replacing $\theta$ with a topologically conjugate epimorphism if neccessary we can assume $d_l\in\Gamma$ for $l>1$ (see proof of Theorem \ref{thm1} in \cite{GMJ}). Suppose that $h\geq 2$; then $d_2\in\Gamma$. We can find a curve $c$ such that $c$ and $y(c)$ form a type 1 standard pair: $c$ is a projection on $\mathbb{H}^2/\Gamma$ of the union of two arcs on $D$, one connecting a point $p\in\alpha_2^*$ to $x_1(d_2(p))\in x_1(\alpha_2)$ and other one connecting $d_2(p)\in\alpha_2$ to $x_1(p)\in x_1(\alpha_2^*)$ (see Figure \ref{figh2} with $l=2$). As $d_2\in\Gamma$ identifies $\alpha_2^*$ with $\alpha_2$ and $x_1d_2x_1^{-1}\in\Gamma$ identifies $x_1(\alpha_2^*)$ with $x_1(\alpha_2)$, $c$ is indeed a two-sided simple closed curve on $\mathbb{H}^2/\Gamma$. The same works if $h=1$ and $d_1\in\Gamma$.
		\begin{figure}[!htbp]\begin{center}
			\includegraphics[scale=.5]{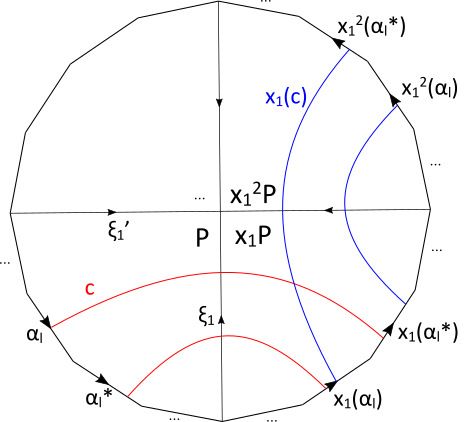}
			\caption{A type 1 standard pair for $d_l\in\Gamma$.}
			\label{figh2}\end{center}
	\end{figure}
	It remains to consider the case when $h=1$ and $d_1\notin\Gamma$. Then by the Riemann-Hurwitz formula $r\geq 2$. If $\theta(\xi_i)\in\{1,...,p-2\}$ for some $i\in\{2,...,r\}$, then there exists $n\in\{2,...,p-1\}$ such that $x_1^nx_i\in\Gamma$. We choose the curve $c$ as the projection of an arc on $D$ connecting the centre of $\xi_i'$ to the centre of $x_1^n(\xi_i)$; since $x_1^nx_i$ identifies these two edges, $c$ is a two-sided closed curve and so is its image $y(c)$; the curves $c$ and $y(c)$ form a type 1 standard pair as in Figure \ref{figpolyr} with $A=\xi_i$. If $\theta(\xi_i)=p-1$ for all $i\in\{2,...,r\}$, then in particular $\theta(x_r)=p-1$ and by composing $\theta$ with the automorphism $\gamma$ defined in Section \ref{topeq} we get $\theta\circ\gamma(x_r)=\theta(x_r^{-1})=-(p-1)=1$. We can then repeat the reasoning with $i=r$.
	\begin{figure}[!htbp]\begin{center}
			\includegraphics[scale=.5]{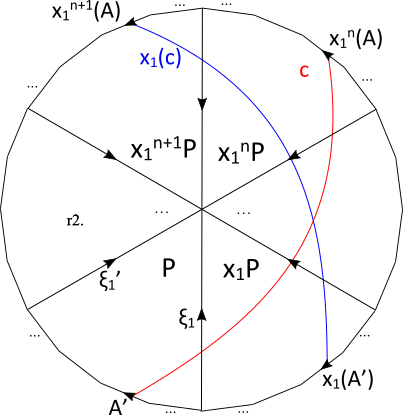}
			\caption{A type 1 standard pair.}
			\label{figpolyr}\end{center}
	\end{figure}
	
	\textbf{Case 3:} $\#f$ is a power of 2. By raising $f$ to an appropriate power we can assume that $\# f=4$. The signature of $\Lambda$ is $\left(h;\pm;[m_1,...,m_r]);\{()^k\}\right)$, where $m_i\in\left\{2,4\right\}$, $h+k>0$ and $k>0$ if the sign is "+" because $\Lambda$ has to contain an orientation-reversing element, and $r+k>0$ since we assume that the action is not free. The marked polygon $P$ associated with $\Lambda$ has the surface symbol $\xi_1\xi_1'...\xi_r\xi_r'\epsilon_1\gamma_{1}\epsilon_1'...\epsilon_k\gamma_{k}\epsilon_k'\alpha_1\beta_1'\alpha_1'\beta_1...\alpha_h\beta_h'\alpha_h'\beta_h$  if the sign is "+" and $\xi_1\xi_1'...\xi_r\xi_r'\epsilon_1\gamma_{1}\epsilon_1'...\epsilon_k\gamma_{k}\epsilon_k'\alpha_1\alpha_1^*...\alpha_h\alpha_h^*$ otherwise.
	
\begin{lem}\label{intlem}
Let $\tilde{y}=x_i$ for some $1\leq i\leq r$ or $\tilde{y}=e_j$ for some $1\leq j\leq k$ and suppose $\theta(\tilde{y})\in\{1,3\}$.
\begin{enumerate}
\item[(a)] If there is an orientation-preserving generator $z$ of $\Lambda$ such that $z\neq\tilde{y}$ and either $\theta(z)=\theta(\tilde{y})$ or $\theta(z)=2$, then there exists a curve $c$ on $N_g=\mathbb{H}^2/\Gamma$ such that, for $y=\tilde{y}\Gamma$, $(c,y(c))$ is a type 1 standard pair.
\item[(b)] If $k\geq 2$ and $\theta(e_s)=0$ for some $1\leq s\leq k$, then there exists a curve $c$ on $N_g=\mathbb{H}^2/\Gamma$ such that, for $y=\tilde{y}\Gamma$, $(c,y(c))$ is a type 2 standard pair.
\end{enumerate}
\end{lem}
\begin{proof} We will consider the cases when $\tilde{y}$ is equal to either $x_1$ or $e_1$.
\item[(a)]First, let $\tilde{y}=x_1$. We can assume that $\theta(x_1)=1$, for otherwise we can compose $\theta$ with an automorphism of $\mathbb{Z}_4$ interchanging 1 and 3. Let $z$ be an orientation-preserving generator of $\Lambda$ which maps the edge $\zeta'$ of the marked polygon $P$ onto $\zeta$. Suppose that $\theta(z)=\theta(\tilde{y})=1$; then $x_1^3z\in\Gamma$ identifies $\zeta'$ with $x_1^3(\zeta)$ and we can find a curve $c$ such that $c$ and $y(c)$ form a type 1 standard pair as in Figure \ref{figpolyr} with $A=\zeta$ and $n=3$. If $\theta(z)=2$, then $x_1^2z\in\Gamma$ identifies $\zeta'$ with $x_1^2(\zeta)$ and we can find a curve $c$ such that $c$ and $y(c)$ form a type 1 standard pair as in Figure \ref{figpolyr} with $A=\zeta$ and $n=2$. Now let $\tilde{y}=e_1$; again, we can assume that $\theta(e_1)=1$. Suppose first that $\theta(z)=\theta(\tilde{y})=1$; then $e_1z^{-1}\in\Gamma$ identifies $e_1(\zeta')$ with $\zeta$ and we can find a curve $c$ such that $c$ and $y(c)$ form a type 1 standard pair as in Figure \ref{figp2e1z1}.
	\begin{figure}[!htbp]\begin{center}
		\includegraphics[scale=.5]{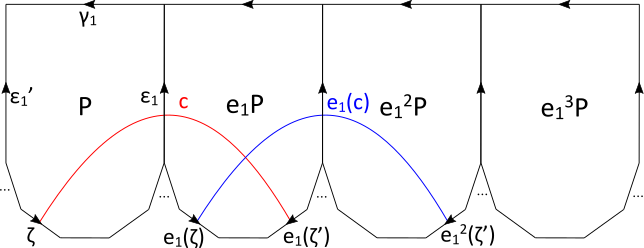}
		\caption{A type 1 standard pair for $\theta(e_1)=1$ and $\theta(z)=1$.}
		\label{figp2e1z1}\end{center}
\end{figure}Now let $\theta(z)=2$; then $e_1^2z\in\Gamma$ identifies $\zeta'$ with $e_1^2(\zeta)$ and we can find a curve $c$ such that $c$ and $y(c)$ form a type 1 standard pair as in Figure \ref{figp2e1z2}.\begin{figure}[!htbp]\begin{center}
		\includegraphics[scale=.5]{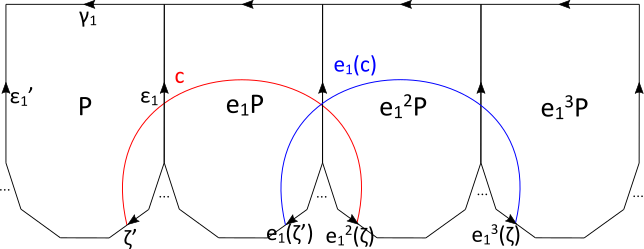}
		\caption{A type 1 standard pair for $\theta(z)=3$ and $\theta(z)=1$.}
		\label{figp2e1z2}\end{center}\end{figure}
\item[(b)]Let $k\geq 2$ and $\theta(e_s)=0$. Because $e_s\in\Gamma$, the image of $\gamma_s$ in $N_g$ is a two-sided closed curve. Let $c$ be this curve. The complement of $c\cup y(c)$ in $N_g$ is the image of the complement of $\gamma_s\cup \tilde{y}(\gamma_s)\cup \tilde{y}^2(\gamma_s)\cup \tilde{y}^3(\gamma_s)$ in $D=P\cup \tilde{y}P\cup\tilde{y}^2P\cup\tilde{y}^3P$ ($\tilde{y}^2c_s$ identifies $\gamma_s$ with $\tilde{y}^2(\gamma_s)$) and is therefore connected. Consider the arc connecting the midpoint of $\gamma_t$ with the midpoint of $y^2(\gamma_t)$ for $t\neq s$. The image of this arc is a one-sided curve on $N_g\setminus (c\cup y(c))$, which is therefore non-orientable. This completes the proof of the lemma.
\end{proof}
	
	Since $\theta$ is an epimorphism, there exists at least one generator of $\Lambda$ which maps onto an element of order 4 in $\mathbb{Z}_{4}$. Without loss of generality we take this element of $\mathbb{Z}_{4}$ to be 1. The generator mapped onto 1 can be one of $x_i$, $i\in\{1,...,r\}$, $e_j$, $j\in\{1,...,k\}$ and $d_l$, $a_l$ or $b_l,$ $l\in\{1,...,h\}$. We will examine the cases for $x_1$, $e_1$, $d_1$ and $a_1$.

	\textbf{Subcase 3.1:} $\theta(x_1)=1$. Our generator of $\Lambda/\Gamma$ is $x_1\Gamma\in\Lambda/\Gamma$. We can write $\Lambda$ as a sum of cosets of $\Gamma$:
	\begin{equation*}
	\Lambda=\Gamma\cup x_1\Gamma\cup x_1^{2}\Gamma\cup x_1^{3}\Gamma
	\end{equation*}
	Then the polygon $D$ which is the fundamental region for $\Gamma$ is
	\begin{equation*}
	P\cup x_1P\cup x_1^{2}P \cup x_1^{3}P=D
	\end{equation*}
	By the long relation there is another generator whose image has order $4$, and it is either $x_i$ for some $i\in\{2,...,r\}$ or $e_j$ for some $j\in\{1,...,k\}$.
	
	If $\#\theta(x_i)=4$, then $\theta(x_i)$ is equal to either 1 or 3. If $\theta(x_i)=1$, we can find a curve $c$ such that $c$ and $y(c)$ form a type 1 standard pair by Lemma \ref{intlem}(a) with $\tilde{y}=x_1$ and $z=x_i$. If $\theta(x_i)=3$ and the sign is "$-$", then by composing $\theta$ with $\rho_i\circ...\circ\rho_{r-1}\circ\gamma$ (see Section \ref{topeq})  we get $\theta\circ\rho_i\circ...\circ\rho_{r-1}\circ\gamma(x_r)=\theta(x_i^{-1})=1$ and we apply Lemma \ref{intlem}(a) with $\tilde{y}=x_1$ and $z=x_r$. If the sign is "$+$" and $h\geq 1$, then by composing $\theta$ with $\rho_i\circ...\circ\rho_{r-1}\circ\sigma^q$, $q\in\mathbb{Z}$, we get $\theta\circ\rho_i\circ...\circ\rho_{r-1}\circ\sigma^q(b_1)=\theta(b_1)+q\cdot\theta(x_i)=\theta(b_1)+q\cdot3$. By choosing the value of $q$ we can set any desired value of $\theta(b_1)$; we set $\theta(b_1)=1$. Then we can once again apply Lemma \ref{intlem}(a), with $z=b_1$.
If the sign is "$+$" and $h=0$, then $k\geq 1$. If there exists an orientation-preserving generator $z$ different from $x_1$ and $x_i$ and such that $\theta(z)\neq 0$, we apply Lemma \ref{intlem}(a), taking either $x_1$ or $x_i$ as $\tilde{y}$. Otherwise $r=2$ and by the Hurwitz-Riemann formula $k\geq 2$ and $\theta(e_k)=0$. We can then apply Lemma \ref{intlem}(b).

	If $\#\theta(e_j)=4$, then $\theta(e_j)$ is equal to either 1 or 3.	If $\theta(e_j)=1$, we can find a curve $c$ such that $c$ and $y(c)$ form a type 1 standard pair by Lemma \ref{intlem}(a) with $\tilde{y}=x_1$ and $z=e_j$. If $\theta(e_j)=3$ and the sign is "$-$", then by composing $\theta$ with $\lambda_j\circ...\circ\lambda_{k-1}\circ\epsilon$ we get $\theta\circ\lambda_j\circ...\circ\lambda_{k-1}\circ\epsilon(e_k)=\theta(e_j^{-1})=1$ and we apply Lemma \ref{intlem}(a) with $z=e_k$. If the sign is "$+$" and $h\geq 1$, then by composing $\theta$ with $\lambda_j\circ...\circ\lambda_{k-1}\circ\pi^q$, $q\in\mathbb{Z}$, we get $\theta\circ\lambda_j\circ...\circ\lambda_{k-1}\circ\pi^q(b_1)=\theta(b_1)+q\cdot\theta(e_j)$. By choosing the value of $q$ we can set any desired value of $\theta(b_1)$; we set $\theta(b_1)=1$. Then we can once again apply Lemma \ref{intlem}(a), with $z=b_1$. If the sign is "$+$" and $h=0$, then by the Riemann-Hurwitz formula $r+k\geq 3$. If there exists an orientation-preserving generator $z$ different from $x_1$ and $e_j$ and such that $\theta(z)\neq 0$, we apply Lemma \ref{intlem}(a), taking either $x_1$ or $e_j$ as $\tilde{y}$. Otherwise $r=1$ and $k\geq 2$, thus $\theta(e_k)=0$ and we apply Lemma \ref{intlem}(b).
	
	\textbf{Subcase 3.2:} $\theta(e_1)=1$.	
Now the chosen generator $y$ is equal to $e_1\Gamma\in\Lambda/\Gamma$. We write $\Lambda$ as a sum of cosets of $\Gamma$:
	\begin{equation*}
	\Lambda=\Gamma\cup e_1\Gamma\cup e_1^{2}\Gamma\cup e_1^{3}\Gamma
	\end{equation*}
	Then the polygon $D$ which is the fundamental region for $\Gamma$ is
	\begin{equation*}
	P\cup e_1P\cup e_1^{2}P \cup e_1^{3}P=D
	\end{equation*}
	As before, there is another generator whose image has order $4$, and it is either $x_i$ for some $i\in\{1,...,r\}$ or $e_j$ for some $j\in\{2,...,k\}$. We assume that $\theta(x_i)=2$ for $1\leq i\leq r$, for otherwise we are in Subcase 3.1. This leaves us with some $e_j$ such that $\#\theta(e_j)=4$.
	
If $\#\theta(e_j)=4$, then $\theta(e_j)$ is equal to either 1 or 3. If $\theta(e_j)=1$, we can find a curve $c$ such that $c$ and $y(c)$ form a type 1 standard pair by Lemma \ref{intlem}(a) with $\tilde{y}=e_1$ and $z=e_j$. If $\theta(e_j)=3$ and the sign is "$-$", then by composing $\theta$ with $\lambda_j\circ...\circ\lambda_{k-1}\circ\epsilon$ we get $\theta\circ\lambda_j\circ...\circ\lambda_{k-1}\circ\epsilon(e_k)=\theta(e_j^{-1})=1$ and we apply Lemma \ref{intlem}(a) with $z=e_k$. If the sign is "$+$" and $h\geq 1$, then by composing $\theta$ with $\lambda_j\circ...\circ\lambda_{k-1}\circ\pi^q$, $q\in\mathbb{Z}$, we get $\theta\circ\lambda_j\circ...\circ\lambda_{k-1}\circ\pi^q(b_1)=\theta(b_1)+q\cdot\theta(e_j)$. By choosing the value of $q$ we can set any desired value of $\theta(b_1)$; we set $\theta(b_1)=1$. Then we can once again apply Lemma \ref{intlem}(a), with $z=b_1$. If the sign is "$+$" and $h=0$, then by the Riemann-Hurwitz formula $r+k\geq 3$. If there exists an orientation-preserving generator $z$ different from $e_1$ and $e_j$ and such that $\theta(z)\neq 0$, we apply Lemma \ref{intlem}(a), taking either $e_1$ or $e_j$ as $\tilde{y}$. Otherwise $r=0$ and $k\geq 3$, thus $\theta(e_k)=0$ and we apply Lemma \ref{intlem}(b).

	\textbf{Subcase 3.3:} $\theta(d_1)=1$.
	Now the chosen generator $y$ is equal to $d_1\Gamma\in\Lambda/\Gamma$. We write $\Lambda$ as a sum of cosets of $\Gamma$:
	\begin{equation*}
	\Lambda=\Gamma\cup d_1\Gamma\cup d_1^{2}\Gamma\cup d_1^{3}\Gamma
	\end{equation*}
	Then the polygon $D$ which is the fundamental region for $\Gamma$ is
	\begin{equation*}
	P\cup d_1P\cup d_1^{2}P \cup d_1^{3}P=D
	\end{equation*}
	By the long relation, there exists another generator not in the kernel, although the image of this generator is not necessarily of order $4$. We assume that $\theta(x_i)=2$ for all $i$, for otherwise we are in Subcase 3.1. Likewise we assume that $\theta(e_j)\in\{0,2\}$ for all $j$; if $\theta(e_j)\in\{1,3\}$ for some $j$, we are in Subcase 3.2.
If $\theta(d_l)=0$ for some $l$, we can find a curve $c$ such that $c$ and $y(c)$ form a type 1 standard pair (Figure \ref{figp2d1d0}); $d_l$ identifies $\alpha_l^*$ with $\alpha_l$, analogously for their images.	\begin{figure}[!htbp]\begin{center}
		\includegraphics[scale=.5]{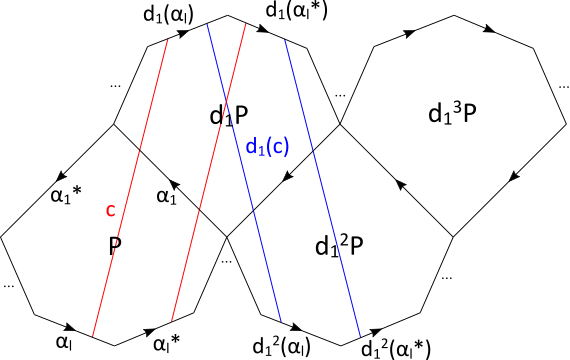}
		\caption{A type 1 standard pair for $\theta(d_1)=1$ and $\theta(d_l)=0$.}
		\label{figp2d1d0}\end{center}
\end{figure}
Similarly, if $\theta(d_l)=2$ for some $l$, we can find a curve $c$ such that $c$ and $y(c)$ form a type 1 standard pair (Figure \ref{figp2d1d2}); $d_1^2d_l\in\Gamma$ identifies $\alpha_l^*$ with $d_1^2(\alpha_l)$ and $d_1d_ld_1^{-3}\in\Gamma$ identifies $d_1^3(\alpha_l^*)$ with $d_1(\alpha_l)$ (recall that $d_1^4\in\Gamma$ identifies $\alpha_1^*$ with $d_1^3(\alpha_1)$). Therefore we can assume that $\theta(d_l)\in\{1,3\}$ for all $l$.
	\begin{figure}[!htbp]\begin{center}
		\includegraphics[scale=.5]{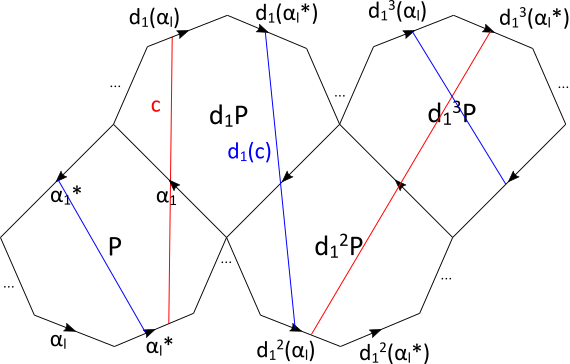}
		\caption{A type 1 standard pair for $\theta(d_1)=1$ and $\theta(d_l)=2$.}
		\label{figp2d1d2}\end{center}
\end{figure}
Then by Lemma \ref{nonorient} $k\geq 1$; otherwise the image of $\Lambda^+$ under $\theta$ is not equal to $\mathbb{Z}_4$.

If $\theta(e_j)=2$ for some $j$, then $e_jd_1^{-2}$ identifies $d_1^2(\epsilon_j')$ with $\epsilon_j$ and we can find a curve $c$ such that $c$ and $y(c)$ form a type 2 standard pair (Figure \ref{figp2d1z2} with $A=\epsilon_j$).
\begin{figure}[!htbp]\begin{center}
		\includegraphics[scale=.5]{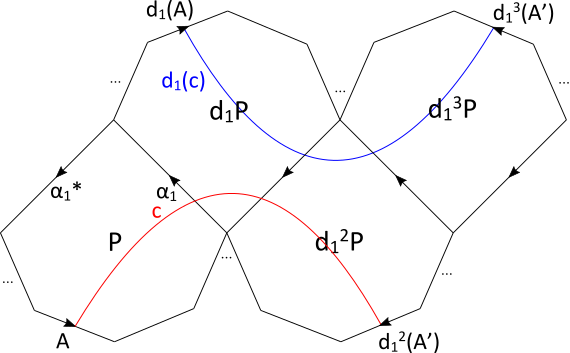}
		\caption{A type 2 standard pair for $\theta(d_1)=1$.}
		\label{figp2d1z2}\end{center}\end{figure}
By the Riemann-Hurwitz formula $r\geq 1$, $k\geq 2$ or $h\geq 2$. If $r \geq 1$ or $k\geq 2$ and assuming $j=1$, there exists an edge ($\xi_1$ or $\gamma_2$) on the segment of $\partial D$ from $A=\epsilon_1$ to $d_1^2(\epsilon_1')$ (anticlockwise) identified with an edge between $d_1^2(\epsilon_1')$ and $d_1^3(\epsilon_1')$ so that the complement of $c\cup y(c)$ is connected. If $h\geq 2$, then by assumption $\theta(d_2)\in\{1,3\}$. If $\theta(d_2)=1$, then the edge $\alpha_2^*$ on the segment from $d_1(\epsilon_1)$ to $\epsilon_1$ is identified with the edge $d_1^3(\alpha_2)$ on the segment from $d_1^3(\epsilon_1')$ to $d_1(\epsilon_1)$. If $\theta(d_2)=3$, then the edge $d_1^3(\alpha_2^*)$ on the segment from $d_1^3(\epsilon_1')$ to $d_1(\epsilon_1)$ is identified with the edge $\alpha_2$ on the segment from $d_1(\epsilon_1)$ to $\epsilon_1$. In any of these cases there exists a curve on $N_g$ intersecting $c\cup y(c)$ in one point, which means that $N_g\setminus (c\cup y(c))$ is connected. A projection of the arc joining the midpoints of $\gamma_1$ and $d_1^2(\gamma_1)$ is a one-sided curve on $N_g\setminus (c\cup y(c))$($d_1^2c_1\in\Gamma$ identifying $\gamma_1$ with $d_1^2(\gamma_1)$) which is therefore non-orientable. From now on we assume $\theta(e_j)=0$ for all $1\leq j\leq k$.

If $k\geq 2$, then the construction in the proof of Lemma \ref{intlem}(b) works for $\tilde{y}=d_1$ (since $d_1^2$ is orientation-preserving). We assume that $k=1$ and by the Riemann-Hurwitz formula $r\geq 2$ or $h\geq 2$. If $r\geq 2$, then we can find a curve $c$ such that  $c$ and $y(c)$ form a type 2 standard pair (Figure \ref{figp2d1z2} with $A=\xi_1$).
There exists an edge (namely $\gamma_1$) in the segment of $\partial D$ from $\xi_1$ to $d_1^2(\xi_1')$ (anticlockwise) which is identified with an edge on the segment from $d_1^2(\xi_1')$ to $d_1^3(\xi_1')$, so that the complement of $c\cup y(c)$ is connected. A projection of the sum of two arcs, one joining the midpoints of $\gamma_1$ and $\xi_2'$ and the other joining the midpoints of $d_1^2(\xi_2)$ and $d_1^2(\gamma_1)$, is a one-sided curve on $N_g\setminus (c\cup y(c))$ ($d_1^2c_1\in\Gamma$ identifying $\gamma_1$ with $d_1^2(\gamma_1)$ and $d_1^2x_2\in\Gamma$ identifying $\xi_2'$ with $d_1^2(\xi_2)$), which is therefore non-orientable. If $h\geq 2$, then by assumption $\theta(d_2)\in\{1,3\}$. If $\theta(d_2)=1$, then $d_1^3d_2\in\Gamma$ identifies $\alpha_2^*$ with $d_1^3(\alpha_2)$ and we can find a curve $c$ such that $c$ and $y(c)$ form a type 1 standard pair (Figure \ref{figp2d1d1}).
\begin{figure}[!htbp]\begin{center}
		\includegraphics[scale=.5]{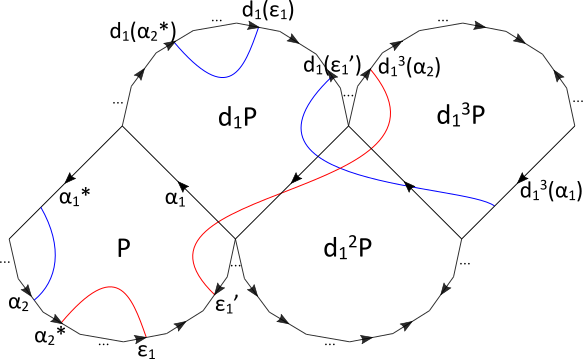}
		\caption{A type 1 standard pair for $\theta(d_1)=1$, $\theta(e_1)=0$ and $\theta(d_2)=1$.}
		\label{figp2d1d1}\end{center}\end{figure}
	If $\theta(d_2)=3$, then $d_1d_2\in\Gamma$ identifies $\alpha_2^*$ with $d_1(\alpha_2)$ and we can find a curve $c$ such that $c$ and $y(c)$ form a type 2 standard pair (Figure \ref{figp2d1d3}). The edge $\gamma_1$ on the segment of $\partial D$ from $\epsilon_1$ to $\epsilon_1'$ is identified with the edge $d_1^2(\gamma_1)$ on the segment from $d_1^2(\alpha_2)$ to $d_1(\epsilon_1')$ and the edge $d_1(\gamma_1)$ on the segment from $d_1(\epsilon_1')$ to $d_1(\epsilon_1)$ is identified with the edge $d_1^3(\gamma_1)$ on the segment from $d_1^2(\alpha_2)$ to $d_1(\epsilon_1')$, ensuring the connectedness of $N_g\setminus (c\cup y(c))$. A projection of the union of two arcs, one connecting the midpoint of $d_1^2(\gamma_1)$ with the midpoint of $d_1^3(\alpha_1)$ and the other connecting the midpoint of $\alpha_1^*$ with the midpoint of $\gamma_1$, is a one-sided curve on $N_g\setminus (c\cup y(c))$, which is therefore non-orientable.
	\begin{figure}[!htbp]\begin{center}
			\includegraphics[scale=.5]{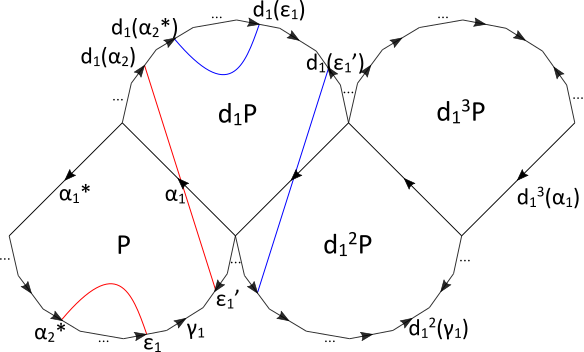}
			\caption{A type 2 standard pair for $\theta(d_1)=1$, $\theta(e_1)=0$ and $\theta(d_2)=3$.}
			\label{figp2d1d3}\end{center}\end{figure}
			
			\textbf{Subcase 3.4:} $\theta(a_1)=1$.
	The chosen generator $y$ is equal to $a_1\Gamma\in\Lambda/\Gamma$. Certainly $k\geq 1$. If there exists $j\in\{1,...,k\}$ such that $e_j\notin\Gamma$, we can assume that $\theta(e_j)=2$ and $a_1^2e_j$ identifies $\epsilon_j'$ with $a_1^2(\epsilon_j)$. Then we can find a curve $c$ such that $c$ and $y(c)$ form a type 1 standard pair (Figure \ref{figp2oa1e2}).
	\begin{figure}[!htbp]\begin{center}
			\includegraphics[scale=.5]{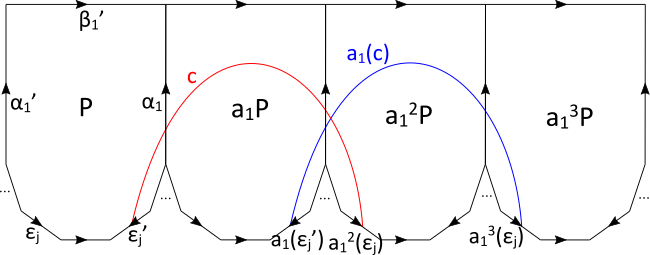}
			\caption{A type 1 standard pair for $\theta(a_1)=1$ and $\theta(e_j)=2$.}
			\label{figp2oa1e2}\end{center}
	\end{figure}
	Now if $e_j\in\Gamma$ for all $j\in\{1,...,k\}$, $e_1$ identifies $\epsilon_1'$ with $\epsilon_1$. If $\theta(b_1)=0$, $b_1$ identifies $\beta_1'$ with $\beta_1$ and we can find a curve $c$ such that $c$ and $y(c)$ form a type 2 standard pair (Figure \ref{figp2oa1e0b0}). The edge $\gamma_1$ on the segment of $\partial D$ from $\epsilon_j$ to $\epsilon_j'$ is identified with the edge $a_1^2(\gamma_1)$ on the segment from $a_1(\epsilon_j')$ to $a_1(\beta_1')$ and the edge $a_1(\gamma_1)$ on the segment from $a_1(\epsilon_j)$ to $a_1(\epsilon_j)$ is identified with the edge $a_1^3(\gamma_1)$ on the segment from $a_1(\epsilon_j')$ to $a_1(\beta_1')$ so that the complement of $c\cup y(c)$ is connected. A projection of the union of two arcs, one connecting the midpoint of $a_1^2(\gamma_1)$ with the midpoint of $a_1^3(\alpha_1)$ and the other connecting the midpoint of $\alpha_1'$ with the midpoint of $\gamma_1$, is a one-sided curve on the complement of $c\cup y(c)$, which is therefore non-orientable. If $\theta(b_1)\neq0$, then by composing $\theta$ with $\omega^q$, $q\in\mathbb{Z}$, we get $\theta\circ\omega^q(b_1)=\theta(b_1)+q\theta(a_1)$ and by choosing the right value of $q$ we can set $\theta\circ\omega^q(b_1)=0$.
		\begin{figure}[!htbp]\begin{center}
			\includegraphics[scale=.65]{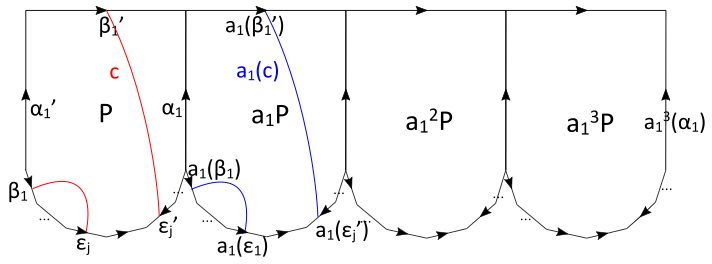}
			\caption{A type 2 standard pair for $\theta(a_1)=1$, $\theta(e_1)=0$ and $\theta(b_1)=0$.}\label{figp2oa1e0b0}\end{center}
	\end{figure}
This completes the proof of Theorem \ref{noninv}.
	\end{proof}
	\section{Classification of involutions}\label{sec:class}
\subsection{Surgeries}
We are going to define five surgeries which will help us construct all the involutions we need in easy steps. We begin with an oriented surface $X$ embedded in $\mathbb{R}^3$ and an involution $\varphi:X\rightarrow X$ induced either by a reflection across some plane in $\mathbb{R}^3$ or by a rotation by the angle $\pi$ about some axis (for example the hyperelliptic involution). Next we apply a series of surgeries that result in our chosen involution. After applying each surgery we obtain a new involution on a new surface $X'$, $\varphi':X'\rightarrow X'$, which we can then further modify.
\begin{description}
\item[Blowing up an isolated fixed point]
Let $x\in X$ be an isolated fixed point, $\varphi(x)=x$. We choose a disc $D\subset X$ such that $x\in D$ and $\varphi (D)=D$. Then $\varphi_{|D}$ is a rotation by the angle $\pi$. Let $X'=(X\setminus \mathrm{int}(D))/\sim$, where $\sim$ identifies antipodal points on $\partial D$. Then $\varphi':X'\rightarrow X'$ is the involution induced by $\varphi$. This surgery replaces an isolated fixed point with a one-sided oval.

\item[Blowing up a non-isolated fixed point] Let $c$ be an oval and $x\in c$. We choose a disc $D\subset X$ such that $x\in D$ and define $X'$ and $\varphi'$ as above. Now $\varphi_{|D}$ is a reflection about $c\cap D$. This surgery preserves the number of ovals, but changes the number of sides of one oval and adds one isolated fixed point.

\item[Blowing up a 2-orbit]
Let $x\in X$, $\varphi(x)\neq x$. We choose a disc $D\subset X$ such that $x\in D\subset X$ and $\varphi(D)\cap D=\emptyset$. Then $X'=X\setminus(\mathrm{int}(D)\cup \mathrm{int}(\varphi(D)))/\sim$, where $\sim$ identifies antipodic points on $\partial D$ and $\partial \varphi(D)$. This surgery adds two crosscaps that map onto one another and increases the genus of the quotient orbifold by 1.

\item[Adding a handle]

Let $x\in X$, $\varphi(x)\neq x$. We choose a disc $D$ as above. $X'=X\setminus (\mathrm{int}(D)\cup \mathrm{int}(\varphi(D)))/\sim$, where $\sim$ identifies $y$ with $\varphi(y)$ for $y\in\partial D$. If $X$ is orientable, then $X'$ is non-orientable if and only if $\varphi$ is orientation-preserving. This surgery adds a two-sided oval.

\item[Surface gluing]

Let $X$, $Y$ be oriented surfaces and $\phi :X\rightarrow X$, $\psi :Y\rightarrow Y$ be involutions, where $\phi$ is orientation-reversing (a reflection) and $\psi$ is orientation-preserving (a rotation). Choose two discs $D\subset X$ and $D'\subset X'$ such that $\phi(D)\cap D=\emptyset$ and $\psi(D')\cap D'=\emptyset$. Let $\theta:\partial D\rightarrow \partial D'$ be an orientation-reversing homeomorphism. Let $X'$ be the surface obtained by removing $\mathrm{int}(D)\cup \mathrm{int}(\phi(D))$ from $X$ and $\mathrm{int}(D')\cup \mathrm{int}(\psi(D'))$ from $Y$ and identifying $\partial D$ with $\partial D'$ via $\theta$ and $\partial \phi(D)$ with $\partial \psi(D') $ via $\psi\circ\theta\circ\varphi^{-1}$. Then $\phi$ and $\psi$ induce an involution on $X'$. Note that $\psi\circ\theta\circ\varphi^{-1}$ is orientation-preserving and $X'$ is non-orientable of genus $2(g(X)+g(Y)+1)$.
\begin{figure}[!htbp]\begin{center}
		\includegraphics[scale=.40]{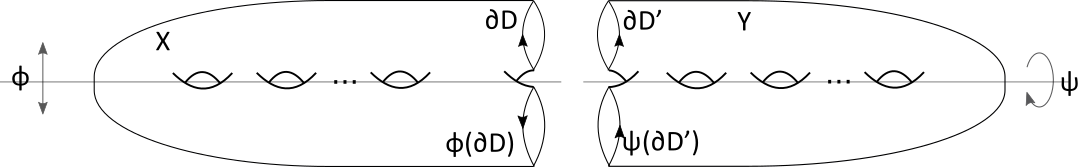}
		\caption{The left-hand side of the surface is reflected, and the right-hand side is rotated.}
	\end{center}
\end{figure}
\end{description}

\subsection{Involutions on $N_g$}\label{sec:invs}
Basing on Dugger's work \cite{Dug} and Theorem \ref{thm2}, we are going to define eleven involutions on a surface $N_g$, which will exhaust all possibilities up to conjugation. We remark that the problem of classification of involutions on compact surfaces is classical, especially in the case of orientable surfaces, but also for non-orientable surfaces several classification results have been obtained by different authors; see \S 1.7 in \cite{Dug}.

Let $f$ be an involution; in this section we do not distinguish between a mapping class and its representative. Let $N_g=\mathbb{H}^2/\Gamma$, $\left\langle f\right\rangle=\Lambda/\Gamma$ and let $\theta:\Lambda\rightarrow \mathbb{Z}_2$ be the canonical projection. The signature of $\Lambda$ is $(h;\pm;[(2)^r];\{()^k\})$.
Recall that $k_+=\#\{j\in\{1,...,k\}|\theta(e_j)=0\}=\#(\{e_1,...,e_k\}\cap\Gamma)$. Then $k_-=\#\{j\in\{1,...,k\}|\theta(e_j)=1\}$.
Further, notice that by the long relation
\begin{equation*}\begin{split}
\theta(x_1...x_re_1...e_k)&= 0\\
\theta(x_1)+...+\theta(x_r)+\theta(e_1)+...+\theta(e_k)&=0
\end{split}
\end{equation*}
which gives
\begin{equation*}
r+k_-\equiv 0\mod(2)
\end{equation*}
and by the Hurwitz-Riemann formula
\begin{equation*}\begin{split}
g-2&=2(\epsilon h+k-2)+r\\
g&=2\epsilon h +2k-2+r\\
g&\equiv r\mod(2)
\end{split}
\end{equation*}
where $\epsilon$ is equal to $1$ if $N_g/\left\langle f \right\rangle$ is non-orientable and $2$ otherwise. Additionally, also by the Hurwitz-Riemann formula, always $r+2k\leq g+2$ and $r+2k\leq g$ if the quotient orbifold is non-orientable.

By Theorem 1.16 of \cite{Dug} the numbers $r$, $k$, $k_+$ and $k_-$ together with information about the orientability of the quotient orbifold (a set of characteristics which the author refers to as the ''signed taxonomy'' of $C_2$-actions) determine an action up to conjugacy if either the surface is orientable, or the quotient orbifold is orientable, or $r+k_->0$. Otherwise, additional invariants are necessary which Dugger terms the $\epsilon$-invariant and the $DD$-invariant. Notice that Theorem 1.16 of \cite{Dug} coincides with Theorem \ref{thm2}.

Let us first consider the case when $r=0$ and $k=0$, that is, the action of $\left\langle f \right\rangle$ is free. Since $k_-=0$, $g$ must be even. This case was already covered in the proof of Theorem \ref{noninv} in Subsection \ref{sec:proofnoninv} (Figure \ref{figfree2}); we have denoted the two involutions $f_{01}$ and $f_{02}$, respectively.
	
Now let $r>0$ and $k=0$. Since $k_-=0$, $r$ must be even, and so must $g$. By Theorem 7.7 of \cite{Dug} (see also Theorem \ref{thm2}) there is exactly one conjugacy class of involutions in this case for fixed $r$, with a non-orientable quotient orbifold. We denote this involution $f_1$. As the representative of this class we can take a rotation by $\pi$ of an orientable surface of genus $r/2-1$ with $h$ 2-orbits blown up, as in Figure \ref{figIn1}.
\begin{figure}[!htbp]\begin{center}
		\includegraphics[scale=.5]{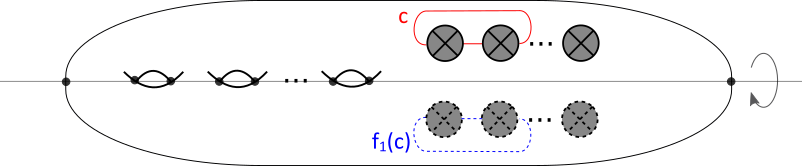}
		\caption{Involution $f_1$ with the set of fixed points marked.}
		\label{figIn1}\end{center}
\end{figure}

Next, let $r>0$, $k_-=0$ and $k_+>0$. Again, $k_-=0$ implies $r$ and $g$ are both even. By Theorem 7.9 of \cite{Dug} (see also Theorem \ref{thm2}) there are two conjugacy classes of involutions for fixed $r$ and $k_+$, one with a non-orientable and one with an orientable quotient orbifold. As a representative of the class with a non-orientable quotient orbifold we can take the reflection of an orientable surface of genus $k-1$ with $h$ 2-orbits blown up and $r$ non-isolated fixed points blown up, as in Figure \ref{figIn2}; we will denote it by $f_2$.
\begin{figure}[!htbp]
	\begin{center}
		\includegraphics[scale=.45]{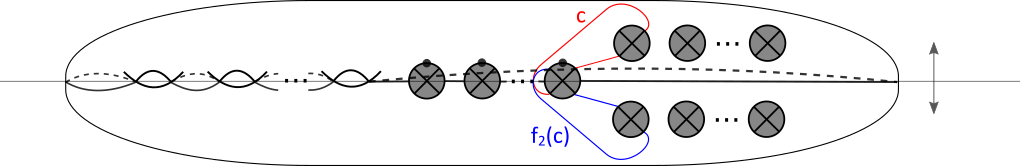}
		\caption{Involution $f_2$ with the set of fixed points marked.}
		\label{figIn2}\end{center}
\end{figure}
As a representative of the class with an orientable quotient orbifold we can take the reflection of an orientable surface of genus $k-1+2h$ with $h$ pairs of conjugate handles and $r$ non-isolated fixed points blown up, as in Figure \ref{figIn3}; we will denote it by $f_3$. By the Hurwitz-Riemann formula, in this case $r+2k\equiv g+2\mod(4)$.
\begin{figure}[!htbp]\begin{center}
		\includegraphics[scale=.45]{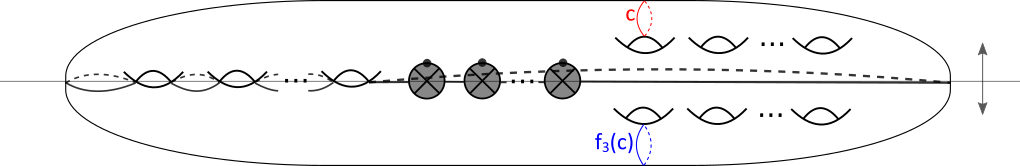}
		\caption{Involution $f_3$ with the set of fixed points marked.}
		\label{figIn3}\end{center}
\end{figure}

Next, let $r\geq 0$, $k_+\geq 0$ and $k_->0$. By Theorem 7.10 of \cite{Dug} (see also Theorem \ref{thm2}) there are two conjugacy classes of such involutions for fixed $r$, $k_-$ and $k_+$, one with a non-orientable and one with an orientable quotient orbifold. As the representative of the class with a non-orientable quotient orbifold we can take a reflection of an orientable surface of genus $k_{+}-1$ glued together with a rotated orientable surface of $k_{-}+r+1$ with $h$ 2-orbits blown up and $k_-$ isolated fixed points blown up (see Figure \ref{figIn4}); we will denote it by $f_4$.
\begin{figure}[!htbp]\begin{center}
		\includegraphics[scale=.35]{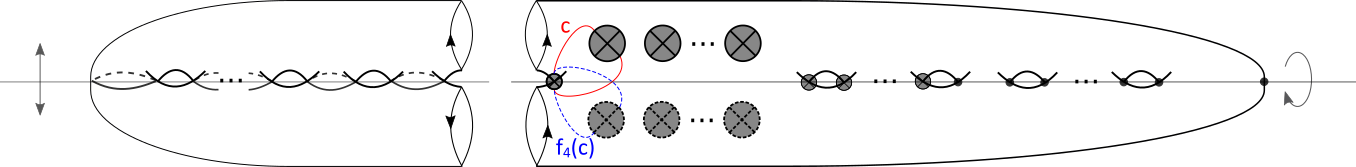}
		\caption{Involution $f_4$ with the set of fixed points marked.}
		\label{figIn4}\end{center}
\end{figure}
As a representative of the class with an orientable quotient orbifold we can take the reflection of an orientable surface of genus $k_{+}-1$ glued together with a rotated orientable surface of genus $k_{-}+r+1+2h$ with $h$ pairs of conjugate handles and $k_-$ isolated fixed point blown up (see Figure \ref{figIn5}); we will denote it by $f_5$.
\begin{figure}[!htbp]\begin{center}
		\includegraphics[scale=.35]{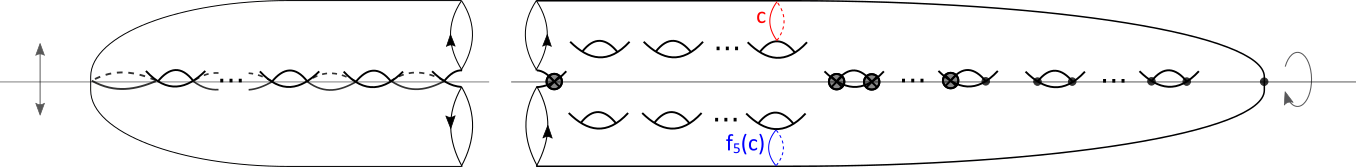}
		\caption{Involution $f_5$ with the set of fixed points marked.}
		\label{figIn5}\end{center}
\end{figure}

Finally let $r=0$, $k\geq 0$ and $k_-=0$. Then $g$ must be even. By Theorem 7.12 of \cite{Dug} (see also Theorem \ref{thm2}) there are four conjugacy classes of such involutions for fixed $k_+$, three with non-orientable quotient orbifolds and one with an orientable one. We will denote the representatives of the conjugacy classes with non-orientable quotient orbifolds by $f_6$, $f_7$ and $f_8$, in order. The involution $f_6$ always occurs. We can take its representative to be the reflection of an orientable surface of genus $k-1$ with $h$ 2-orbits blown up (Figure \ref{figIn6}). Note that in this case the set of fixed points is separating.
\begin{figure}[!htbp]\begin{center}
		\includegraphics[scale=.5]{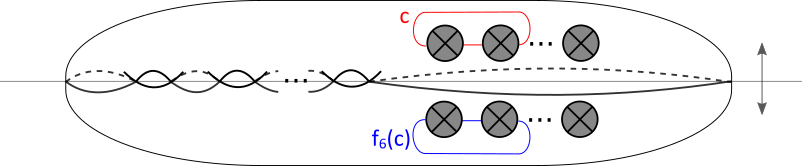}
		\caption{Involution $f_6$ with the set of fixed points marked.}
		\label{figIn6}\end{center}
\end{figure}
The involution $f_7$ occurs if and only if $2k<g$. We can take its representative to be the antipodism of a sphere with $h-1$ 2-orbits blown up and $k$ added handles (Figure \ref{figIn7}). Notice that the condition $2k<g$ is necessary for the surface to be non-orientable.
\begin{figure}[!htbp]\begin{center}
		\includegraphics[scale=.5]{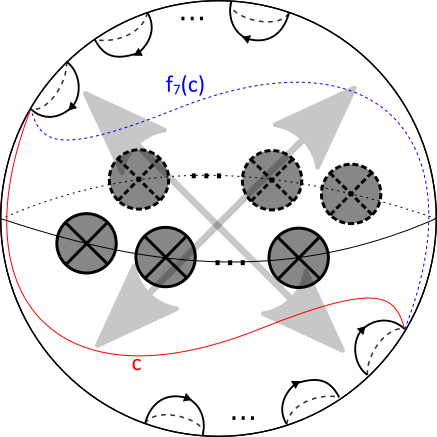}
		\caption{Involution $f_7$ with the set of fixed points marked. The two-sided ovals result from adding $k$ handles -- a surgery consisting in identification of boundary points on $k$ pairs of antipodal discs whose interiors are removed from the surface. Three of these discs are drawn.}
		\label{figIn7}\end{center}
\end{figure}
The involution $f_8$ occurs if and only if $2k<g-2$. We can take its representative to be the antipodism of a torus with $h-2$ 2-orbits blown up and $k$ added handles (Figure \ref{figIn8}). Notice that the condition $2k<g-2$ is necessary for the surface to be non-orientable.
\begin{figure}[!htbp]\begin{center}
		\includegraphics[scale=.5]{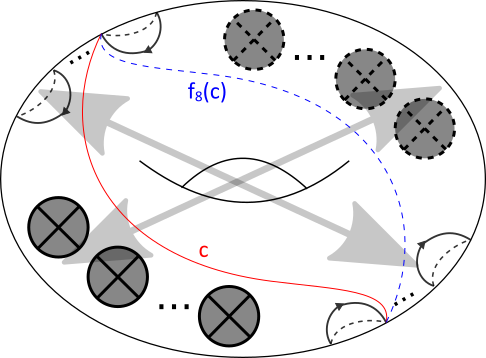}
		\caption{Involution $f_8$ with the set of fixed points marked. The two-sided ovals result from adding $k$ handles.}
		\label{figIn8}\end{center}
\end{figure}
The conjugacy class of involutions with an orientable quotient orbifold occurs if and only if $2k\equiv g+2\mod 4$, which condition is necessitated by the orientability of the quotient orbifold. We can take its representative to be the rotation by $\pi$ of an orientable surface of genus $2h+1$ with $k$ added handles about the axis passing through its centre (Figure \ref{figIn9}); we denote it by $f_9$.
\begin{figure}[!htbp]\begin{center}
		\includegraphics[scale=.5]{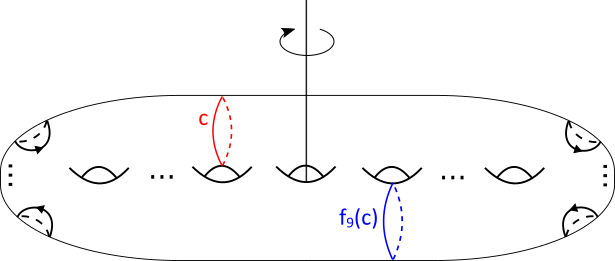}
		\caption{Involution $f_9$ with the set of fixed points marked. The two-sided ovals result from adding $k$ handles.}
		\label{figIn9}\end{center}
\end{figure}

\begin{rem}
If $k=0$, $f_7$ and $f_8$ coincide with $f_{01}$ and $f_{02}$ (Figure \ref{figfree2}).
\end{rem}

Summarising, from Theorems 7.7, 7.9, 7.12 in \cite{Dug} we obtain the following lemma.

\begin{lem}
The involutions $f_1,...,f_9$ described above are the only, up to conjugacy, involutions on a non-orientable surface $N_g$.
\end{lem}

\section{Normal closures of involutions}\label{sec:proofinv}

\subsection{Actions on homology groups}\label{hom}
We know from Lemma \ref{L2.2} that the existence of a curve $c$ such that $c$ and $f(c)$ form a standard pair of type 1 or 2 is a sufficient condition for the normal closure of $f$ to contain the commutator subgroup. We will need a way to show that a condition is also necessary, that is, a method of showing that the normal closure of a map does not contain the commutator subgroup.

Let $V_g=H_1(N_g;\mathbb{Z}_2)\cong \mathbb{Z}_2^g$. The standard generators for the homology group of a non-orientable surface of genus $2h+k$ are the homology classes of the curves $a_i,b_i,i\in\{1,...,h\}$ and $c_j,j\in\{1,...,k\}$ on Figure \ref{figvg}, where $g=2h+k$. We will denote by $[a]$ the homology class of a curve $a$ in $V_g$.

We have an intersection form $\left\langle \cdot ,\cdot \right\rangle : V_g \times V_ g\longrightarrow \mathbb{Z}_2$, where 
\begin{align*}
\left\langle [c_i],[c_j] \right\rangle& =\delta_{ij},&
\left\langle [a_i],[a_j] \right\rangle& =0,&
\left\langle [b_i],[b_j] \right\rangle& =0,\\
\left\langle [a_i],[b_j] \right\rangle& =\delta_{ij},&
\left\langle [a_i],[c_j] \right\rangle& =0,&
\left\langle [b_i],[c_j] \right\rangle& =0.
\end{align*}
$\mathcal{M}(N_g)$ acts on $V_g$ preserving $\left\langle \cdot , \cdot \right\rangle$.

Let $V_g^+=\left\{[h]\in V_g : \left\langle [h],[h] \right\rangle=0\right\}\cong\mathbb{Z}_2^{g-1}$. Notice that $V_g^+$ is generated by $[a_i],[b_i],i\in\{1,...,g\}$ and $[c_i]+[c_j],i,j\in\{1,...,k\},i\neq j$. Note that $[c]=[c_1]+...+[c_k]$ is the only element of $V_g$ such that for any $[h]\in V_g$ $\left\langle [h],[c] \right\rangle=\left\langle [h],[h] \right\rangle$, which implies that any element $f\in \mathcal{M}(N_g)$ preserves $[c]$. If $k$ is even, $[c]$ lies in $V_g^+$ and $\mathcal{M}(N_g)$ acts on $V_g^+/\left\langle [c]\right\rangle\cong \mathbb{Z}_2^{g-2}$.

Let $a$, $b$ be two-sided curves on $N_g$ such that $i(a,b)=1$. The automorphism of $V_g$ induced by $T_aT_b$ is non-trivial, as it maps $[b]$ onto $[a]+[b]$; the induced automorphisms of $V_g^+$ and $V_g^+/\langle c \rangle$ are likewise non-trivial. Since $T_aT_b\in[\mathcal{M}(N_g),\mathcal{M}(N_g)]$ by Lemma \ref{L2.1}, it follows that the normal closure of an element acting trivially on any of the spaces $V_g$, $V_g^+$ and $V_g^+/\langle c \rangle$ does not contain the commutator subgroup of $\mathcal{M}(N_g)$. We will present three examples and trust the reader to apply them in further cases.

\begin{ex}\label{ex1}

We take the reflection of an orientable surface of genus $h$ with $k$ non-isolated fixed points blown up (Figure \ref{figvg}).
\begin{figure}[!htbp]\begin{center}
\includegraphics[scale=.5]{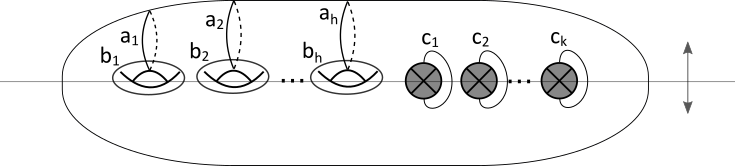}
\caption{Surface $N_g$ in Example \ref{ex1} with standard homology basis.}
\label{figvg}\end{center}
\end{figure}
This involution acts on generators of $V_g$ as follows: $[a_i]\mapsto [a_i]$, $[b_i]\mapsto [b_i]$, $[c_j]\mapsto [c_j]$. The action on $V_g$ is trivial. It is easy to show that the same holds for the hyperelliptic involution of an orientable surface with $k$ isolated points blown up.

\end{ex}

\begin{ex}\label{ex2}

We take the hyperelliptic involution of an orientable surface of genus $h$ with one 2-orbit blown up (Figure \ref{figvg+}).
\begin{figure}[!htbp]\begin{center}
\includegraphics[scale=.5]{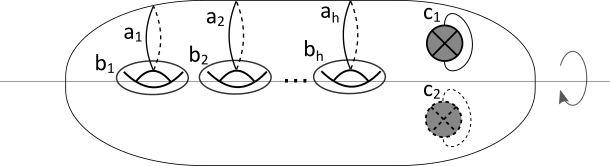}
\caption{Surface $N_g$ in Example \ref{ex2} with standard homology basis.}\label{figvg+}\end{center}
\end{figure}
This involution acts on generators of $V_g$ as follows: $[a_i]\mapsto [a_i]$, $[b_i]\mapsto [b_i]$, $[c_1]\mapsto [c_2]$, $[c_2]\mapsto [c_1]$. The action on $V_g$ is not trivial, so we look to $V_g^+=\left\langle [a_i],[b_i],[c_1]+[c_2] \right\rangle$. The involution takes $[c_1]+[c_2]$ to $[c_2]+[c_1]=[c_1]+[c_2]$, and the action on $V_g^+$ is trivial. An analogous result for the reflection of a surface with one 2-orbit blown up is easily obtained.

\end{ex}

\begin{ex}\label{ex3}

We take the reflection of an orientable surface of genus $h$ with two 2-orbits blown up (Figure \ref{figvg+c}).
\begin{figure}[!htbp]\begin{center}
\includegraphics[scale=.5]{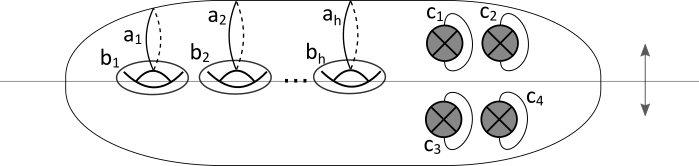}
\caption{Surface $N_g$ in Example \ref{ex3} with standard homology basis.}\label{figvg+c}\end{center}
\end{figure}
As above, the action on $V_g$ is not trivial. We have $V_g^+=\left\langle [a_i],[b_i],[c_1]+[c_2],[c_1]+[c_3],[c_1]+[c_4]\right\rangle$ and the action is as follows: $[a_i]\mapsto [a_i]$, $[b_i]\mapsto [b_i]$, $[c_1]+[c_2]\mapsto [c_3]+[c_4]$, $[c_1]+[c_3]\mapsto [c_1]+[c_3]$, $[c_1]+[c_4]\mapsto [c_2]+[c_3]$. The action on $V_g^+$ is not trivial, but notice that $[c_3]+[c_4]=[c_1]+[c_2]\mod(c)$ and $[c_2]+[c_3]=[c_1]+[c_4]\mod(c)$, which means that the action on $V_g^+/\left\langle [c]\right\rangle$ is trivial. An analogous result for the hyperelliptic involution of a surface with two 2-orbits blown up is easily obtained.

\end{ex}

\begin{rem}

It is clear that if $K_1$ is the kernel of the action of $\mathcal{M}(N_g)$ on $V_g$, $K_2$ is the kernel of the action of $\mathcal{M}(N_g)$ on $V_g^+$ and $K_3$ is the kernel of the action of $\mathcal{M}(N_g)$ on $V_g^+/\left\langle [c]\right\rangle$, we have $K_1\lneqq K_2\lneqq K_3$. Indeed, the examples we have shown prove that there exist involutions in each of $K_1$, $K_2\setminus K_1$ and $K_3\setminus K_2$.

\end{rem}

\subsection{Normal closures}
\begin{proof}[Proof of Theorem \ref{inv}.]

Let $f\in\mathcal{M}(N_g)$ be an involution, $r$, $k$, $k_+$, $k_-$ the parameters describing the structure of the set of fixed points of $f$ and $h$ the genus of $N_g/\langle f \rangle$.

\textbf{Case 1:} $r>0$ and $k=0$. The only involution for $r>0$ and $k=0$, up to conjugacy, is $f_1.$ It is easy to see that if $h\geq 3$, we can find a curve $c$ such that $c$ and $f_1(c)$ form a type 2 standard pair, as in Figure \ref{figIn1}.
Notice that the genus of the surface is equal to $g=2(r/2-1)+2h=r+2h-2$, and so the condition in terms of $g$ and $r$ is $g-r=2h-2\geq 4$.
Now let us consider the case with $h=2$. Then the induced map $f_{1*}$ acts trivially on the space $V_{g}^{+}/\left\langle c \right\rangle$ (see Example \ref{ex3} in Section \ref{hom}), and therefore the normal closure of $f_1$ cannot contain the commutator subgroup. If $h=1$, the action is trivial on the space $V_g^+$ (see Example \ref{ex2}). $h=0$ does not occur. The condition is necessary.

\textbf{Case 2:} $k>0$ and $r+k_->0$.

\textbf{Subcase 2a:} $N_g/\langle f \rangle$ is non-orientable. The involutions satisfying these conditions and resulting in a non-orientable surface orbifold are $f_2$ and $f_4$. For the involution $f_2$ we can find a curve $c$ such that $c$ and $f_2(c)$ form a type 1 standard pair for $h\geq 1$ (see Figure \ref{figIn2}). By the definition of $f_2$ the orbifold $N_g/\left\langle f_2 \right\rangle$ is non-orientable, hence $h\geq 1$ is always satisfied. For the involution $f_4$ we can find a curve $c$ such that $c$ and $f_4(c)$ form a type 1 standard pair if $h\geq 1$  (see Figure \ref{figIn4}). By the definition of $f_4$, $k_->0$ and the orbifold $N_g/\left\langle f_4 \right\rangle$ is non-orientable, hence $h\geq 1$ is also always satisfied.

\textbf{Subcase 2b:} $N_g/\langle f \rangle$ is orientable. The involutions satisfying these conditions and resulting in an orientable surface orbifold are $f_3$ and $f_5$.
For the involution $f_3$ we can find a curve $c$ such that $c$ and $f_3(c)$ form a type 2 standard pair if $h\geq 1$ (see Figure \ref{figIn3}).
Since $g=2(k-1)+r+4h$, this translates to $g-r-2k=4h-2\geq 2$. If $h=0$, $f_{3*}$ acts trivially on the space $V_g$ (see Example \ref{ex1}), hence the condition is necessary.
Likewise, for the involution $f_5$ we can find a curve $c$ such that $c$ and $f_5(c)$ form a type 2 standard pair if $h\geq 1$ (see Figure \ref{figIn5}).
Since $g=2(k-1)+r+4h$, this translates to $g-r-2k=4h-2\geq 2$. If $h=0$, $f_{5*}$ acts trivially on the space $V_g$ (see Figure \ref{figIn5h}). Therefore the condition is necessary.
\begin{figure}[!htbp]\begin{center}\includegraphics[scale=.4]{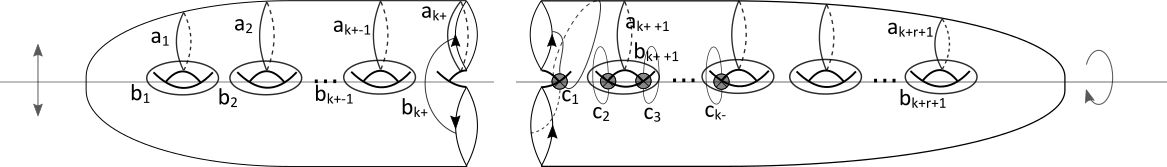}
\caption{Involution $f_5$ on the surface with standard homology base for $h=0$.}
\label{figIn5h}\end{center}
\end{figure}

\textbf{Case 3:} $r=k_{-}=0$.

\textbf{Subcase 3a:} the set of fixed points is separating. This is satisfied by the involution $f_6$.
For the involution $f_6$ we can find a curve $c$ such that $c$ and $f_6(c)$ form a type 2 standard pair if $h-1\geq 3$ (see Figure \ref{figIn6}). Since $g=2h+2(k-1)$, this translates to $g-2k=2h-2\geq 4$. If $h=2$, the induced map $f_{6*}$ acts trivially on the space $V_{g}^{+}/\left\langle c \right\rangle$, and if $h=1$, the action is trivial on the space $V_g^+$; $h=0$ does not occur. The condition is necessary (see Examples \ref{ex3} and \ref{ex2}).

\textbf{Subcase 3b:} the set of fixed points is non-separating. This is satisfied by the involutions $f_7$ and $f_8$ when the orbifold is non-orientable and $f_9$ when the orbifold is orientable.
The case when $k=0$ and the involution is free was shown in the proof of Theorem \ref{noninv} in Section \ref{sec:proofnoninv}; we now assume $k>0$.
For the involution $f_7$ we can find a curve $c$ such that $c$ and $f_7(c)$ form a type 1 standard pair for $h-1\geq 1$ (see Figure \ref{figIn7}), which is always satisfied, as the surface is nonorientable (by the definition of $f_7$, $h-1$ is the number of pairs of crosscaps on $N_g$).
For the involution $f_8$ we can find a curve $c$ such that $c$ and $f_8(c)$ form a type 1 standard pair for $h-2\geq 1$ (see Figure \ref{figIn8}), which is also always safisfied for analogous reasons.
For the involution $f_9$ we can find a curve $c$ such that $c$ and $f_9(c)$ form a type 1 standard pair if $2h+1\geq 1$ (see Figure \ref{figIn9}). This is always satisfied. This completes the proof of Theorem \ref{inv}.
\end{proof}
\begin{proof}[Proof of Theorem \ref{thm:normgens}.]
The involution normally generating $\mathcal{M}(N_g)$ will be a variant of $f_4$ with $k_+=0$, which can be realised as a rotation by $\pi$ of an orientable surface of genus $(k_-+r-2)/2$ with $h$ 2-orbits blown up and $k_-$ isolated fixed points blown up. We will denote it by $f$.

	By Theorems \ref{inv} and \ref{twist} the normal closure of $f$ contains the twist subgroup $\mathcal{T}(N_g)$. In order to apply Theorem \ref{thm:twsb} to show that $f$ does not lie in the twist subgroup, we calculate the action of the induced map $f_{*}$ on the homology group $H_1(N_g;\mathbb{R})$.
	
		\begin{figure}[!htbp]\begin{center}
			\includegraphics[scale=.45]{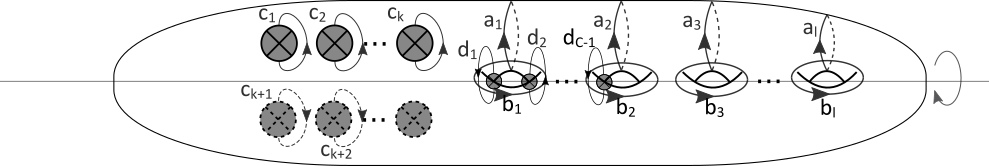}
			\caption{Involution $f$ with standard base for $H_1(N_g;\mathbb{R})$.}
			\label{figf4ch}\end{center}
	\end{figure}

The action is
\begin{equation*}
\begin{split}
a_i\mapsto & -a_i-2\sum_{j=2i}^{k_-}d_j,\quad i=1,...,\lceil k_-/2\rceil-1\\
a_i\mapsto & -a_i,\quad i=\lceil k_-/2\rceil,...,l\\
b_i\mapsto & -b_i-2(d_{2i-1}+d_{2i}),\quad i=1,...,\lceil k_-/2\rceil\\
b_i\mapsto & -b_i,\quad i=\lceil k_-/2\rceil+1,...,l\\
c_i\mapsto & c_{i+h},\quad i=1,...,h-1\\
c_h\mapsto & -\sum_{i=1}^{2h-1}c_i-\sum_{j=1}^{k_-}d_j\\
c_{i+h}\mapsto & c_i,\quad i=1,...,h-1\\
d_j\mapsto & d_j,\quad j=1,...,k_--1
\end{split}
\end{equation*}

The determinant of $F_*$ is equal to $(-1)^h$. By Theorem \ref{thm:twsb} $f\notin\mathcal{T}(N_g)$ if $h$ is odd. For $N_g$ we can construct $f$ to have any chosen $h\geq 1$ by adjusting $r$ and $k_-$ according to the Hurwitz-Riemann formula, including any odd $h$. This completes the proof of Theorem \ref{thm:normgens}.
\end{proof}

\section*{Acknowledgements}

This paper is part of the author's Ph.D. thesis, written under the supervision of B{\l}a\.zej Szepietowski at the University of Gdańsk. The author wishes to express her gratitude for the supervisor's helpful suggestions.

\end{document}